\documentclass[10pt]{amsart}
\usepackage{amsfonts}
\usepackage{indentfirst,amsthm,bm,amsmath}
\usepackage{hyperref}

\newtheorem{theorem}{Theorem}[section]
\newtheorem{corollary}[theorem]{Corollary}
\newtheorem{lemma}[theorem]{Lemma}
\newtheorem{proposition}[theorem]{Proposition}
\newtheorem{definition}[theorem]{Definition}
\newtheorem{remark}[theorem]{Remark}

\newtheorem{conjecture}[theorem]{Conjecture}

\begin{document}

\title[Second main theorem and defect relation]{Second main theorem with tropical hypersurfaces and defect relation}
\date{}
\author[Cao Tingbin]{Cao Tingbin}
\address[Cao Tingbin]{Department of Mathematics, Nanchang University, Nanchang, Jiangxi 330031, P.R. China}
\email{tbcao@ncu.edu.cn}
\thanks{The first author is supported by the National Natural Science Foundation of China (\#11461042) and the outstanding young talent assistance program of Jiangxi Province (\#20171BCB23002) in China.}

\author[ Zheng Jianhua]{Zheng Jianhua}
\address[Zheng Jianhua]{Department of Mathematics, Tsinghua University, Beijing 100084, P.R. China}
\email{jzheng@math.tsinghua.edu.cn}
\thanks{The second author is supported by the National Natural Science Foundation of China (\#11571193).}



\begin{abstract}The tropical Nevanlinna theory is Nevanlinna theory for tropical functions or maps over the max-plux semiring by using the approach of complex analysis. The main purpose of this paper is to study the second main theorem with tropical hypersurfaces into tropical projective spaces and give a defect relation which can be regarded as a tropical version of the Shiffman's conjecture. On the one hand, our second main theorem improves and extends the tropical Cartan's second main theorem due to Korhonen and Tohge [Advances Math. 298(2016), 693-725]. The growth of tropical holomorphic curve is also improved to $\limsup_{r\rightarrow\infty}\frac{\log T_{f}(r)}{r}=0$ (rather than just hyperorder strictly less than one) by obtaining an improvement of tropical logarithmic derivative lemma. On the other hand, we obtain a new version of tropical Nevanlinna's second main theorem which is different from the tropical Nevanlinna's second main theorem obtained by Laine and Tohge [Proc. London Math. Soc. 102(2011), 883-922]. The new version of the tropical Nevanlinna's second main theorem implies an interesting defect relation that $\delta_{f}(a)=0$ holds for a nonconstant tropical meromorphic function $f$ with $\limsup_{r\rightarrow\infty}\frac{\log T_{f}(r)}{r}=0$ and any $a\in\mathbb{R}$ such that $f\oplus a\not\equiv a.$
\end{abstract}
\maketitle

\section{Introduction}
Recently the so-called ¡°tropical approach¡± to mathematics has attracted much attention from researchers in several
fields
including combinatorics, optimization, mathematical physics and algebraic. Surprisingly, many results from classical
algebraic geometry have tropical analogues. The tropical geometry (see \cite{mikhalkin} or a recent book
 \cite{maclagan-sturmfels}) is a piecewise linear version of complex algebraic geometry, which is geometry over
 the tropical semiring (that is, the max-plus semi-ring first arose in Kleene¡¯s 1956 paper on nerve sets and
 automata \cite{kleene}).\par

The tropical Nevanlinna theory can be seen as Nevanlinna theory over
the tropical semiring by using the approach of complex analysis.
Functions or maps are naturally defined on tropical semiring. In
\cite{halburd-southall}, Halburd and Southall first proved the
tropical versions of Poisson-Jensen theorem, Nevanlinna first main
theorem by introducing tropical versions of Nevanlinna
characteristic function, proximity function, counting function,
logarithmic derivative lemma for tropical meromorphic functions with
finite order similarly as in the classical Nevanlinna theory (see
for examples \cite{hayman, cherry-ye, ru}). Using this theory, they
investigated the existence of finite-order max-plus meromorphic
solutions which can be considered to be an ultra-discrete analogue
of the Painlev\'{e} property, in which they proposed a tropical
version of Clunie theorem. Laine and Tohge \cite{laine-tohge}
considered tropical Laurent series with arbitrary real value powers
and extended the definition of tropical meromorphic functions, then
in this case the multiplicities of poles (respectively, zeros) may
be arbitrary real numbers instead of being integers (respectively,
rationals) which is in certain respects fundamentally different from
the counterparts in the classical meromorphic functions. After
modifying the Halburd-Southall's version of Tropical Poisson-Jensen
theore and first main theorem, Laine and Tohge gave the tropical
Nevanlinna second main theorem for tropical meromorphic functions
with hyperorder strictly less than one. Laine, Liu and Tohge
\cite{laine-liu-tohge} presented tropical counterparts of some
classical complex results related to Fermat type equations, Hayman
conjecture and Br\"{u}ck conjecture.  Recently, Korhonen and Tohge
\cite{korhonen-tohge-2016} extended the tropical Nevanlinna theory
to tropical holomorphic curves in a finite dimensional tropical
projective space, and obtained a tropical version of Cartan second
main theorem for tropical holomorphic curves with hyperorder
strictly less than one. \par

The main purpose of this paper is to study the second main theorem
for tropical holomorphic curves intersecting tropical hypersurfaces
(Theorem \ref{T1}) and then obtain a defect relation (Theorem
\ref{T5}) which can be regarded as a tropical version of the
Shiffman's conjecture in Classical Nevanlinna theorem completely
proved by Ru \cite{ru2004}. The growth condition for tropical
holomorphic curves is improved to the case
$\limsup_{r\rightarrow\infty}\frac{\log T_{f}(r)}{r}=0$ (rather than
just hyperorder  strictly less than one), by obtaining an
improvement of the tropical logarithmic derivative lemma (Theorem
\ref{L1}) due to \cite{halburd-southall, laine-tohge}. As shown in
our main result (Theorem \ref{T1}), we obtain an equality form of
second main theorem under the assumption $ddg(\{P_{M+2}\circ f,
\ldots, P_{q}\circ f\})=0$ so that the Korhonen and Tohge's version
of tropical Cartan second main theorem \cite{korhonen-tohge-2016} is
improved. Moreover, from the improvement of tropical Cartan second
main theorem, we obtain new versions of tropical Nevanlinna second
main theorem (Theorem \ref{T3} and Theorem \ref{T6}) which are
different from the Laine and Tohge's version of tropical second main
theorem, and thus propose an interesting defect relation:
$\delta_{f}(a)=0$ holds for a nonconstant tropical meromorphic
function $f$ with $\limsup_{r\rightarrow\infty}\frac{\log T_{f}(r)
}{r}=0$ and any $a\in \mathbb{R}$ such that $f\oplus a\not\equiv a.$
It is very interesting that the assumption $f\oplus a_{j}\not\equiv
a_{j}$ in our results are clearly better than the assumptions in
Laine and Tohge's version of tropical Nevanlinna second main
theorem. During the study of the tropical Nevanlinna second main
theorem, we find that there exists one gap in the proof of Korhonen
and Tohge's improvement of tropical Nevanlinna second main theorem
due to Laine and Tohge.\par

This paper is organized as follows. In the next section, we will
briefly introduce some definitions and notations of the tropical
semiring, topical linear algebra, Nevanlinna theory for tropical
meromorphic functions, and tropical holomorphic curves into tropical
projective spaces. After proving a strong lemma for a nondecreasing
convex continuous function, we give an improved tropical version of
logarithmic derivative lemma (Theorem \ref{L1}) in Section 3. In
Section 4, we firstly show that the identity of the two definitions
of tropical algebraically nondegenerated and algebraically
independently in the Gondran-Minoux sense, and then propose mainly
the second main theorem with tropical hypersurfaces (Theorem
\ref{T1}). This main result improves the tropical Cartan second main
theorem (Corollary \ref{C0}). After establishing the tropical first
main theorem for tropical holomorphic curve intersecting tropical
hypersurfaces (Theorem \ref{T4}) similarly as in Classical
Nevanlinna theory, we then obtain the defect relation (Theorem
\ref{T5}) according to Theorem \ref{T1}. The proof of our main
result (Theorem \ref{T1}) is individually proved in Section 5. The
new versions of tropical Nevanlinna second main theorem (Theorem
\ref{T3} and Theorem \ref{T6}) will be discussed in Section 6, from
which we also obtained an interesting result on defect relation for
tropical meromorphic functions (Theorem \ref{C1}).

\section{Preliminaries}
According to \cite{pin} the term "tropical" appeared in computer science in honor of Brazil and,
more specifically, after Imre Simon (who is a Brazilian computer scientist) by Dominique Perrin.
In computer science the term is usually applied to $(\min, +)-$semiring. Here we use the tropical
(max-plus) semiring $(\mathbb{R}\cup\{-\infty\}, \oplus, \odot)$ endowing $\mathbb{R}\cup\{-\infty\}$
with (tropical) addition
$$x\oplus y:=\max(x,y)$$ and (tropical) multiplication $$x\odot y:=x+y.$$ But the two semirings
are isomorphic each other by $x\mapsto -x.$ We also use notations
$x\oslash y:=\frac{x}{y}\oslash=x-y$ and $x^{\odot a}:=ax$ for $a\in
\mathbb{R}.$ The identity elements for the tropical operations are
$0_{o}=-\infty$ for addition and $1_{o}=0$ for multiplication (see
for examples, \cite{mikhalkin, halburd-southall, laine-tohge}).\par

\begin{definition}\cite{korhonen-laine-tohge} A tropical entire function $f: \mathbb{R}\rightarrow\mathbb{R}$
is a finite or infinite linear combination of tropical monomials:
\begin{eqnarray*}
f(x):=\bigoplus_{k=0}^{+\infty}c_{k}\odot x^{\odot s_{k}}
=\max_{k=0}^{+\infty}\{s_{k}x+c_{k}\}
\end{eqnarray*}
where $c_{1}, c_{2}, \ldots$ and $s_{1}, s_{2}, \ldots$ are real numbers such that
$c_{1}>c_{2}>\ldots(\rightarrow -\infty)$ and $s_{0}<s_{1}<\cdots.$ The tropical polynomial
$P: \mathbb{R}\rightarrow\mathbb{R}$ defined as a finite linear combination of tropical monomials.
\end{definition}
Remark  that a tropical entire function $f: \mathbb{R}\rightarrow\mathbb{R}$ has the following three important
properties:\par

(i) $f$ is continuous,\par
(ii)$f$ is piecewise-linear, where the number of pieces is finite or infinite,\par
(iii)$f$ is convex, that is $f(\frac{x+y}{2})\leq\frac{f(x)+f(y)}{2}$ for all $x, y\in\mathbb{R}.$\par

\begin{definition} \cite{korhonen-laine-tohge} A continuous piecewise linear function
$f:\mathbb{R}\rightarrow\mathbb{R}$ is called to be tropical meromorphic. A point $x\in\mathbb{R}$
of derivative discontinuity of $f$ is a pole of $f$ with multiplicity $-\omega_{f}(x)$ whenever
$$\omega_{f}(x):=\lim_{\varepsilon\rightarrow 0+}\{f'(x+\varepsilon)-f'(x-\varepsilon)\}<0,$$
and a zero (or root) with multiplicity $\omega_{f}(x)$ if $\omega_{f}(x)>0.$\end{definition}

Note that the multiplicities are positive numbers, not necessarily
integers as it is the case in classical complex analysis. From the
viewpoint of geometry, we see that a zero (or pole) of $f$ is the
point $x\in\mathbb{R}$ at which the graph of $f$ is nonlinear and
convex (or concave). Korhonen and Tohge \cite[Proposition
3.3]{korhonen-tohge-2016} proved that for any tropical meromorphic
function $f,$ there exist two tropical entire functions $g$ and $h$
such that $f=h\oslash g,$ where $g$ and $h$ do not have any common
zeros.  \par

The tropical proximity function for tropical meromorphic functions $f$ in one real variable is defined by
\begin{equation*}m(r,f):=\frac{1}{2}\sum_{\sigma=\pm 1}f^{+}(\sigma r)=\frac{1}{2}\{f^{+}(r)+f^{+}(-r)\}
\end{equation*}where $f^{+}(x)=\max\{f(x), 0\}.$ Denote by $n(r,f)$ the number of poles of $f$, counted with
multiplicities, in the interval $(-r, r),$ i.e.,
$$n(r,f)=\sum_{|b_{\nu}|<r}\tau_{f}(b_{\nu})$$
and define the tropical counting function \begin{equation*}
N(r,f):=\frac{1}{2}\int_{0}^{r}n(t,
f)dt=\frac{1}{2}\sum_{|b_{\nu}|<r}\tau_{f}(b_{\nu})(r-|b_{\nu}|)
\end{equation*} where all $b_{\nu}$ are poles of $f.$ The tropical Nevanlinna characteristic function is given by
\begin{equation*}T(r,f):=m(r,f)+N(r,f)
\end{equation*} which is then an increasing convex function of $r$ by the tropical Cartan identity
\cite[Theorem 3.8]{korhonen-laine-tohge}. For example, a nonconstant tropical rational function satisfies
$T(r, f)=O(r)$ (see \cite[Theorem 1.14]{korhonen-laine-tohge}). The tropical Poison-Jensen formula
\cite{halburd-southall, laine-tohge} (see also \cite[Theorem 3.1]{korhonen-laine-tohge}) implies the
tropical Jensen formula
\begin{equation}\label{E1.0}
N(r,1_{0}\oslash f)-N(r,f)= \frac{1}{2}\sum_{\sigma=\pm 1}f(\sigma r)-f(0),
\end{equation}which is very useful and will be used many times throughout this paper. Furthermore,
there exists the tropical first main theorem (see \cite[Theorem 3.5]{korhonen-laine-tohge})
\begin{equation} \label{E1.1}T(r, 1_{o}\oslash (f\oplus a))=T(r, f)+O(1)\end{equation}
provided that $-\infty<a<L_{f}:=\inf\{f(b): \omega_{f}(b)<0\}.$\par

Since the tropical linear space $L$ in $\mathbb{R}^{n+1}$ is closed
under tropical scalar multiplication,
$L=L+\mathbb{R}(1,1,\ldots,1).$ Therefore the tropical projective
space is defined as
$\mathbb{TP}^{n}=\mathbb{R}^{n+1}/\mathbb{R}(1,1,\ldots,1).$ That
is, $\mathbb{TP}^n=\mathbb{R}_{\max}^{n+1}\setminus\{0_{o}\}/\sim$
where $(a_{0}, a_{1}, \ldots, a_{n})\sim (b_{0}, b_{1}, \ldots,
b_{n})$ if and only if
$$(a_{0}, a_{1}, \ldots, a_{n})=\lambda\odot (b_{0}, b_{1}, \ldots, b_{n})
=(\lambda\odot b_{0}, \lambda\odot b_{1}, \ldots, \lambda\odot b_{n})$$
for some $\lambda\in \mathbb{R}.$ Denote by $[a_{0}: a_{1}: \cdots: a_{n}]$
the equivalence class of $(a_{0}, a_{1}, \ldots, a_{n}).$ For instance, $\mathbb{TP}^{1}$
is just  the completed max-plus semiring $\mathbb{R}_{\max}\cup\{+\infty\}=\mathbb{R}\cup\{\pm\infty\}.$\par

Let $f:=[f_{0}: f_{1}:\cdots: f_{n}]:\mathbb{R}\rightarrow \mathbb{TP}^{n}$ be a tropical holomorphic map where $f_{0}, f_{1}, \ldots, f_{n}$ are tropical entire functions and do not have any zeros which are common to all of them. Denote $\mathbf{f}=(f_{0}, f_{1}, \ldots, f_{n}):\mathbb{R}\rightarrow \mathbb{R}^{n+1}.$ Then the map $\mathbf{f}$ is called a reduced representation of the tropical holomorphic curve $f$ in $\mathbb{TP}^{n}.$ The tropical Cartan characteristic function of $f$ is defined by
\begin{equation*}T_{f}(r):=\frac{1}{2}\sum_{\sigma=\pm 1}\|f(\sigma r)\|-\|f(0)\|=\frac{1}{2}[\|f(r)\|+\|f(-r)\|]-\|f(0)\|
\end{equation*}where $\|f(x)\|=\max\{f_{0}(x), \ldots, f_{n}(x)\}$ is defined according to canonical coordinates of $\mathbb{TP}^{n}.$   It is well defined, that is, $T_{f}(r)$ is independent of the reduced representation of $f$ \cite[Proposition 4.3]{korhonen-tohge-2016}. The order and hyperorder of $f$ are given by
\begin{eqnarray*}
\rho(f)=\limsup_{r\rightarrow\infty}\frac{\log T_{f}(r)}{\log r},
\end{eqnarray*}and
\begin{eqnarray*}
\rho_{2}(f)=\limsup_{r\rightarrow\infty}\frac{\log\log T_{f}(r)}{\log r},
\end{eqnarray*}
respectively. If $f$ is a tropical meromorphic function, then $T_{f}(r)=T(r,f)+O(1)$
(see \cite[Proposition 4.4]{korhonen-tohge-2016}).
\par

The operations of tropical addition $\oplus$ and tropical multiplication $\odot$ for the $(n+1)\times(n+1)$
matrices $A=(a_{ij})$ and $B=(b_{ij})$ are defined by $$A\oplus B=(a_{ij}\oplus b_{ij})$$ and
$A\odot B=\left(\bigoplus _{k=0}^{n}a_{ik}\odot b_{kj}\right),$ respectively. If an $(n+1)\times(n+1)$
matrix $A$ contains at least one element different from $0_{o}$ in each row, then $A$ is called regular.
The tropical determinant $|A|_{o}$ of $A$ is defined by \begin{equation*}
 |A|_{o}=\bigoplus a_{0\pi(0)}\odot a_{1\pi(1)}\odot\cdots\odot a_{n\pi(n)},
 \end{equation*}where the sum is taken over all permutations $\{\pi(0), \pi(1), \ldots, \pi(n)\}$ of $\{0, 1, \ldots, n\}.$ Note that an $(n+1)\times(n+1)$ matrix $A$ is regular if and only if $|A|_{0}\neq 0_{o}.$\par

Choose $c\in\mathbb{R}\setminus\{0\}.$ Let $f:\mathbb{R}\rightarrow\mathbb{TP}^{n}$  be a tropical holomorphic map with reduced representation $(f_{0}, f_{1}, \ldots, f_{n}).$ We use short notations $$\overline{f_{j}}^{[0]}:=f_{j}(x),\quad \overline{f_{j}}^{[1]}:=f_{j}(x+c), \quad \overline{f_{j}}^{[k]}:=f_{j}(x+kc)=f_{j}(x\odot c^{\odot k})$$ for all $j, k\in\{0, 1, \ldots, n\}.$ The tropical Casorati determinant, or tropical Casoratian, of $f$ is defined by
\begin{eqnarray*}C_{o}(f):=C_{o}(f_{0}, f_{1}, \ldots, f_{n})=\bigoplus \overline{f_{0}}^{[\pi(0)]}\odot \overline{f_{1}}^{[\pi(1)]}\odot\cdots\odot\overline{f_{n}}^{[\pi(n)]}
\end{eqnarray*} where the sum is taken over all permutations $\{\pi(0), \ldots, \pi(n)\}$ of $\{0, \ldots, n\}.$\par

Tropical meromorphic functions $g_{0}, \ldots, g_{n}$ are linearly dependent (respectively independent)
in the Gondran-Minoux sense \cite{gondran-minoux-1, gondran-minoux-2} if there exist (respectively there
do not exist) two disjoint subsets $I$ and $J$ of $K:=\{0, \ldots, n\}$ such that $I\cup J=K$ and
\begin{eqnarray*} \bigoplus _{i\in I}a_{i}\odot g_{i} =\bigoplus _{j\in J}a_{j}\odot g_{j},
\end{eqnarray*}that is,
\begin{eqnarray*} \max_{i\in I}\{a_{i}+g_{i}\} =\max_{j\in J}\{a_{j}+g_{j}\},
\end{eqnarray*}where the constants $a_{0}, a_{1}, \ldots, a_{n}\in\mathbb{R}_{\max}$ are not all equal to $0_{o}.$
If $a_{0}, \ldots, a_{n}\in\mathbb{R}_{\max}$ and $f_{0}, \ldots, f_{n}$ are tropical entire functions, then
\begin{eqnarray*}
F=\bigoplus_{\nu=0}^{n} a_{\nu}\odot f_{\nu}=\bigoplus_{i=1}^{j}a_{k_i}\odot f_{k_i}
\end{eqnarray*} is called a tropical linear combination of $f_{0}, f_{1}, \ldots, f_{n}$ over $\mathbb{R}_{\max},$
where the index set $\{k_{1}, \ldots, k_{j}\}\subset\{0, \ldots,
n\}$ is such that $a_{k_i}\in \mathbb{R}$ for all $i\in\{1, \ldots,
j\},$ while $a_{\nu}=0_{o}$ if $\nu\not\in\{k_{1}, \ldots, k_{j}\}.$
Note that if $f_{0}, \ldots, f_{n}$ are linearly independent in the
sense of Gondran and Minous, then the express of $F$ cannot be
rewritten by means of any other index set which is different from
the set $\{k_{1}, \ldots, k_{j}\}.$ \par

Let $G=\{f_{0}, \ldots, f_{n}\}(\neq \{0_{o}\})$ be a set of tropical entire functions, linearly independent in
the Gondran-Minoux sense, and denote \begin{equation*}\mathcal{L}_{G}=span<f_{0}, \ldots, f_{n}>
=\left\{\bigoplus_{k=0}^{n}a_{k}\odot f_{k}: (a_{0}, \ldots, a_{n})\in\mathbb{R}^{n+1}_{\max}\right\}
\end{equation*} to  be their linear span. The collection $G$ is called the spanning basis of $\mathcal{L}_{G}.$
The dimension of $\mathcal{L}_{G}$ is defined by
$$\dim(\mathcal{L}_{G})=\max\{\ell(F): F\in \mathcal{L}_{G}\setminus\{0_{o}\}\},$$ where $\ell(F)$ is the
shortest length of the representation of $F\in \mathcal{L}_{G}\setminus\{0_{o}\}$ defined by
$$\ell(F)=\min\{j\in\{1,\ldots, n+1\}: F=\bigoplus_{i=1}^{j}a_{k_i}\odot f_{k_i}\}$$ where
$a_{k_i}\in\mathbb{R}$ with integers $0\leq k_1<k_2<\cdots <k_j\leq
n.$ Note that usually the dimension of the tropical linear span
space of $G$ may not be n+1, which is different from the classical
linear algebraic. If $\ell(F)=n+1$ for a tropical linear combination
$F$ of $f_{0}, \ldots, f_{n},$ then $F$ is said to be complete, that
is, the coefficients $a_{k}$ in any expression of $F$ of the form
$F=\bigoplus_{k=0}^{n}a_{k}\odot f_{k}$ must satisfy $a_{k}\in
\mathbb{R}$ for all $k\in\{0, \ldots, n\}$ and in this case,
$\mathcal{L}_{G}=n+1.$\par

Let $G=\{f_{0}, \ldots, f_{n}\}$ be a set of tropical entire functions, linearly independent in the Gondran-Minoux
sense, and let $Q\subset\mathcal{L}_{G}$ be a collection of tropical linear combinations of $G$ over
$\mathbb{R}_{\max}.$ The degree of degeneracy of $Q$ is defined to be $$ddg(Q):=card (\{F\in Q: \ell(F)<n+1\}).$$
If $ddg(Q)=0,$ then we say $Q$ is non-degenerate. This means that the degree of degeneracy of a set of tropical
linear combinations is the number of its non-complete elements. In this way the number of complete elements of
$Q$ is the 'actual dimension' of the subspace spanned by $Q,$ and thus the $ddg(Q)$ is the `codimension' of the
subspace spanned by $Q$ (see \cite[Page 120-121]{korhonen-laine-tohge}). \par

\section{Tropical version of logarithmic derivative lemma}
In the classical Nevanlinna theory, the logarithmic derivative lemma
plays a key role in significant expression of the Nevanlinna second
main theorem for meromorphic functions and Cartan second main
theorem for holomorphic curve intersecting hyperplanes. The tropical
analogue of the lemma on the logarithmic derivative for tropical
meromorphic functions with finite order was obtained by Halburd and
Southall \cite{halburd-southall}, and was extended to the case of
hyper-order strictly less than one by Laine and Tohge
\cite{laine-tohge}. In this section, we will improve the tropical
logarithmic derivative lemma, and extend the condition of growth of
meromorphic functions to the case not exceeding to the hyperorder 1
minimal type, i.e., $\limsup_{r\rightarrow\infty} \frac{\log
T_{f}(r)}{r}=0$ (rather than just hyperorder strictly less than
one). This result will be used to prove our tropical second main
theorem with tropical hypersurfaces in next section.
\par

\begin{theorem}(Tropical version of logarithmic derivative lemma)\label{L1} Let $c\in\mathbb{R}\setminus\{0\}.$
If $f$ is a tropical meromorphic function on $\mathbb{R}$ with
\begin{equation}\label{E-8.0} \limsup_{r\rightarrow\infty}
\frac{\log T_{f}(r)}{r}=0,
\end{equation}  then
\begin{equation*}
m(r, f(x+c)\oslash f(x))=o\left(T_{f}(r)\right)
\end{equation*} where $r$ runs to infinity outside of a set of zero upper density measure $E,$ i.e.,
$$\overline{dens}E=\limsup_{r\rightarrow\infty}\frac{1}{r}\int_{E\cap [1,r]}dt=0.$$
\end{theorem}

\begin{remark}\label{R0}(i). We note that the condition \eqref{E-8.0} implies that $\rho_{2}(f)\leq 1$ and
the equality can possibly take happened. In fact, assume that
\eqref{E-8.0} holds, then there exists $r_{0}>0$ such that for any
$r>r_{0},$ we have $\log T_{f}(r)<r$ and thus $\rho_{2}(f)\leq 1.$
Moreover, whenever $f$ is taken to satisfy, for example $\log
T_{f}(r)=\frac{r}{(\log r)^{m}}$ where $m> 0,$  one can easily get
both \eqref{E-8.0} and $\rho_{2}(f)=1.$ Hence, Theorem \ref{L1} is
an improvement of the tropical logarithmic derivative lemma obtained
in \cite{halburd-southall, laine-tohge} for hyperorder strictly less
than one.\par

(ii). By the improved version of tropical logarithmic derivative lemma,
the tropical Clunie and Mohon'ko type theorems \cite[Corollaries 4.4, 4.11, 4.12 and 7.15]{korhonen-laine-tohge}
(see also \cite{halburd-southall, laine-tohge, laine-yang}) can be also improved to the case
$\limsup_{r\rightarrow \infty}\frac{\log T_{f}(r)}{r}=0.$
\end{remark}

Before giving the proof, we show the following lemma which is an
improvement of a result on growth properties of nondecreasing
continuous real functions (\cite[Lemma 2.1]{halburd-korhonen2007}
and \cite{halburd-korhonen-tohge2014}). Here we establish a strong
lemma of real convex functions for improving the tropical
logarithmic derivative lemma, an analogue result for a logarithmic
convex function is obtained in \cite{ZK}.\par

\begin{lemma}\label{L2} Let $T(r)$ be a nondecreasing positive, convex, continuous function on $[1, +\infty)$ with
\begin{equation*}
\liminf_{r\rightarrow\infty} \frac{\log T(r)}{r}=0.
\end{equation*} Then for the function $$\phi(r):=\max_{1\leq t\leq
r}\left\{\left(\frac{t}{\log T(t)}\right)^{\delta}\right\}, \quad
\delta\in(0, \frac{1}{2}),$$ we have
\begin{equation*}
T(r)\leq T(r+\phi(r))\leq (1+\varepsilon(r))T(r),
\end{equation*} where $\varepsilon(r)\to 0$
as $r$ tends to infinity outside of a set of zero lower density measure $E,$ i.e.,
$$\underline{dens}E=\liminf_{r\rightarrow\infty}\frac{1}{r}\int_{E\cap [1,r]}dt=0;$$\par
Especially, for any  fixed positive real value $c(\neq 0),$
\begin{equation*}
T(r)\leq T(r+c)\leq (1+\varepsilon(r))T(r),\ r\not\in E\to\infty.
\end{equation*}

Furthermore, if the growth assumption is changed into  \begin{equation*}
\limsup_{r\rightarrow\infty} \frac{\log T(r)}{r}=0,
\end{equation*}
then the exceptional set $E$ is a set with zero upper density
measure, i.e.,
$$\overline{dens}E=\limsup_{r\rightarrow\infty}\frac{1}{r}\int_{E\cap
[1,r]}dt=0.$$
\end{lemma}

\begin{proof} Since $T(r)$ is a nondecreasing positive, convex, continuous function on $[1, +\infty),$
it follows that for any $r\in[1, \infty)$ and $R=r+\phi(r),$ we have
\begin{eqnarray*}
T(R)&\leq&T(r)+\frac{dT(R)}{dR}\left(R-r\right)\\
&\leq& T(r)+\frac{T^{'}(r+\phi(r))}{T(r+\phi(r))} \phi(r) T(r+\phi(r)).
\end{eqnarray*}
Define
\begin{eqnarray*}
\hat{\tau}(r)=\sqrt{\frac{\log T(r+\phi(r))-T(1+\phi(1))}{r}}, \,\, r\in[1, +\infty).
\end{eqnarray*} Since \begin{equation*}
\liminf_{r\rightarrow\infty} \frac{\log T(r)}{r}=0,
\end{equation*} there exists one sequence $r_{n}$ with $\lim_{n\rightarrow \infty}r_{n}=\infty$ such that
\begin{eqnarray*}
\hat{\tau}(r_{n})=\min_{1\leq t\leq r_{n}}\hat{\tau}(t).
\end{eqnarray*} This gives \begin{eqnarray*}
0<\hat{\tau}(r_{n+1})\leq \hat{\tau}(r_{n})\rightarrow 0
\end{eqnarray*} as $n\rightarrow \infty.$ Define $\tau(r):=\hat{\tau}(r_{n})$ for $r\in[r_{n}, r_{n+1}]$
and so $\tau(r)$ tends to zero as $r$ tends to $\infty.$ Now
consider the set
$$E:=\{r\in[1, +\infty): \frac{T'(r+\phi(r))}{T(r+\phi(r))}\geq
\tau(r)\}.$$ Then by the Riemann-Stieljies integral, it follows that
\begin{eqnarray*}
&&\log T(r+\phi(r))-\log T(1+\phi(1))\\&\geq&\int_{1}^{r}{\rm d}(\log T(t+\phi(t)))\\
&=& \int_{1}^{r}\frac{T'(t+\phi(t))}{T(t+\phi(t))}{\rm d}(t+\phi(t))\\
&=&\int_{1}^{r}\frac{T'(t+\phi(t))}{T(t+\phi(t))}{\rm d}t
+\int_{1}^{r}\frac{T'(t+\phi(t))}{T(t+\phi(t))}{\rm d}\phi(t)\\
&\geq& \int_{1}^{r}\frac{T'(t+\phi(t))}{T(t+\phi(t))}{\rm d}t\\
&\geq& \int_{E\cap [1, r)}\tau(t){\rm d}t\\
&\geq& \tau(r)\int_{E\cap [1,r]}{\rm d}t.
\end{eqnarray*} Hence for $r=r_n$,
\begin{eqnarray*}
\frac{1}{r}\int_{E\cap [1,r]}dt&\leq&\frac{1}{\hat{\tau}(r_{n})}\frac{\log T(r_{n}+\phi(r_{n}))-\log T(1+\phi(1))}{r_{n}}\\
&\leq&\sqrt{\frac{\log T(r_{n}+\phi(r_{n}))-\log T(1+\phi(1))}{r_{n}}},
\end{eqnarray*}
and thus
\begin{eqnarray*}\underline{dens}E=\liminf_{r\rightarrow\infty}\frac{1}{r}\int_{E\cap [1,r]}dt=0.\end{eqnarray*}
Therefore, for all $r\not\in E,$ we have
\begin{eqnarray*}
T(r+\phi(r))&\leq&T(r)+\tau(r)\phi(r)T(r+\phi(r))
\end{eqnarray*} which implies
\begin{equation*}T(r+\phi(r))\leq (1+\varepsilon(r))T(r)\end{equation*}where
$\varepsilon(r):=\frac{\tau(r)\phi(r)}{1-\tau(r)\phi(r)}.$ Since $\delta\in(0, \frac{1}{2}),$
we get that $\tau(r)\phi(r)\rightarrow 0,$ and thus $\varepsilon(r)\rightarrow 0,$ as $r\rightarrow\infty.$\par

By noting that $\phi(r)\to\infty$ as $r\to\infty$, we have, for a
fixed positive real number $c$, $T(r+c)\leq T(r+\phi(r))$. Therefore
we obtain the desired result.

Obviously, if the growth assumption is, instead, $
\limsup_{r\rightarrow\infty} \frac{\log T(r)}{r}=0,$ then
the above sequence $\{r_{n}\}$ can be arbitrarily chosen, and thus
the exceptional set $E$ is a set with zero upper density
measure.\par
\end{proof}

\begin{proof}[Proof of Theorem \ref{L1}.] By \cite[Theorem 3.24]{korhonen-laine-tohge},
for all $\alpha(>1)$ and all $r>0,$ we have
\begin{eqnarray*}
m(r, f(x+c)\oslash f(x))\leq\frac{16|c|}{r+|c|}\frac{1}{\alpha-1} T_{f}(\alpha(r+|c|))+\frac{|f(0)|}{2}.
\end{eqnarray*}  Let \begin{equation*}
\alpha:=1+\frac{(r+|c|)^{\delta-1}}{(\log T_{f}(r+|c|))^{\delta}}, \,\,\, \delta(0, \frac{1}{2}).
\end{equation*}
Under the assumption \eqref{E-8.0}, we get that
\begin{eqnarray*}
\frac{1}{(\alpha-1)(r+|c|)}=\left(\frac{\log
T_{f}(r+|c|)}{r+|c|}\right)^{\delta}=o(1) (r\to\infty).
\end{eqnarray*} Note that \begin{eqnarray*}
\alpha(r+|c|)=(r+|c|)+\left(\frac{r+|c|}{\log
T_{f}(r+|c|)}\right)^{\delta}.\end{eqnarray*} Take $\phi(r)=\max_{1\leq t\leq r}\{\left(\frac{t}{T_{f}(t)}\right)^{\delta}\}$  in
Lemma \ref{L2} and so
$\left(\frac{r+|c|}{T_{f}(r+|c|)}\right)^{\delta}\leq\phi(r+|c|)$.
We get that
$$T_{f}(\alpha(r+|c|))\leq(1+\varepsilon(r+|c|))T_{f}(r+|c|)\leq(1+\varepsilon(r+|c|))(1+\varepsilon(r))T_{f}(r)$$
holds for all $r$ with $r\not\in E$ where $\overline{dens}E=0.$
Therefore, we have
\begin{eqnarray*}
m(r, f(x+c)\oslash f(x))&\leq& 16|c|o(1) (1+\varepsilon(r+|c|))(1+\varepsilon(r)) T_{f}(r) +\frac{|f(0)|}{2}\\
&=&o(T_{f}(r))
\end{eqnarray*} holds for all $r\not\in E$ where $\overline{dens}E=0.$
\end{proof}

In terms of the Hinkkanen's Borel type growth lemma, in another way,
we establish the following

\begin{theorem}\label{C2} Let $c\in\mathbb{R}\setminus\{0\}.$
If $f$ is tropical meromorphic function on $\mathbb{R}$ with \begin{equation}\label{E-8.0'}
\limsup_{r\rightarrow\infty} \frac{\log T_{f}(r)(\log r)^{\varepsilon}}{r}=0
\end{equation} holds for any $\varepsilon>0,$ then
\begin{equation*}
m(r, f(x+c)\oslash f(x))=o\left(T_{f}(r)\right)
\end{equation*} where $r$ runs to infinity outside of a set $E$ satisfying
\begin{eqnarray*} \int_{E}\frac{dt}{t\log t}<+\infty.\end{eqnarray*}
\end{theorem}

The next lemma is the Hinkkanen's Borel type growth lemma (or see also a similar lemma
\cite[Lemma 3.3.1]{cherry-ye}.\par

\begin{lemma}\cite[Lemma 4]{hinkkanen} \label{L3} Let $p(r)$ and $h(r)=\varphi(r)/r$ be
positive nondecreasing functions defined for $r\geq\varrho>0$ and
$r\geq\tau>0,$ respectively, such that
$\int_{\varrho}^{\infty}\frac{dr}{p(r)}=\infty$ and
$\int_{\tau}^{\infty}\frac{dr}{\varphi(r)}<\infty.$ Let $u(r)$ be a
positive nondecreasing function defined for $r\geq r_{0}\geq\varrho$
such that $u(r)\rightarrow\infty$ as $r\rightarrow\infty.$ Then if
$C$ is real with $C>1,$ we have
$$u\left(r+\frac{p(r)}{h(u(r))}\right)<Cu(r)$$ whenever $r\geq r_{0},$
$u(r)>\tau,$ and $r\not\in E$ where
$$\int_{E}\frac{dr}{p(r)}\leq\frac{1}{h(w)}+\frac{C}{C-1}\int_{w}^{\infty}\frac{dr}{\varphi(r)}<\infty$$
and $w=\max\{\tau, u(r_{0})\}.$
\end{lemma}

Now we give the proof of Theorem \ref{C2}.\par

\begin{proof}[Proof of Theorem \ref{C2}.] In Lemma \ref{L3},
take \begin{equation*}u(r)=T_{f}(r),\ p(r)=r\log r,\
h(r)=\frac{\varphi(r)}{r},\end{equation*} where $$\varphi(r)=r\log r
(\log\log  r)^{1+\varepsilon}$$ with the $\varepsilon>0$ given by
the assumption of theorem. Then it is obvious that
$\int_{\varrho}^{\infty}\frac{dr}{p(r)}=\infty$ and
$\int_{\tau}^{\infty}\frac{dr}{\varphi(r)}<\infty$ for
$r\geq\varrho>0$ and $r\geq\tau>0.$ Let \begin{equation*}
\alpha:=1+\frac{p(r+|c|)}{(r+|c|)h(T_{f}(r+|c|))}=1+\frac{\log
(r+|c|)}{\log T_{f}(r+|c|)(\log\log T_{f}(r+|c|))^{1+\varepsilon}}.
\end{equation*}
Note that $$T_{f}(\alpha(r+|c|))=T_{f}\left(r+|c|+\frac{(r+|c|)\log
(r+|c|)}{\log T_{f}(r+|c|)(\log\log
T_{f}(r+|c|))^{1+\varepsilon}}\right).$$ Applying Lemma \ref{L3}, we
have
\begin{equation}\label{E-8.2}
T_{f}(\alpha(r+|c|))\leq C T_{f}(r+|c|)
\end{equation} for all $r$ possibly outside a set $E$ satisfying
\begin{eqnarray*}&\ &
\int_{E}\frac{dt}{p(t)}=\int_{E}\frac{dt}{t\log
t}\\&\leq&\frac{1}{\log w(\log \log
w)^{1+\varepsilon}}+\frac{C}{C-1}\int_{w}^{\infty}\frac{dt}{t \log
t(\log\log t)^{1+\varepsilon}}\\&<&+\infty.\end{eqnarray*}

By \cite[Theorem 3.24]{korhonen-laine-tohge}, for all $\alpha(>1)$ and all $r>0,$ we have
\begin{eqnarray*}
m(r, f(x+c)\oslash f(x))\leq\frac{16|c|}{r+|c|}\frac{1}{\alpha-1} T_{f}(\alpha(r+|c|))+\frac{|f(0)|}{2}.
\end{eqnarray*}
Thus we get from \eqref{E-8.2} that for the above defined $\alpha,$
\begin{eqnarray}\label{E-8.1}&&
m(r, f(x+c)\oslash f(x))\\\nonumber&\leq& 16|c|\frac{\log T_{f}(r+|c|)(\log (r+|c|))^{\varepsilon}}{r+|c|}
\left(\frac{\log\log T_{f}(r+|c|)}{\log (r+|c|)}\right)^{1+\varepsilon}
T_{f}(\alpha(r+|c|))\\\nonumber&&+\frac{|f(0)|}{2}
\\\nonumber&\leq& 16|c|\frac{\log T_{f}(r+|c|)(\log (r+|c|))^{\varepsilon}}{r+|c|}
\left(\frac{\log\log T_{f}(r+|c|)}{\log (r+|c|)}\right)^{1+\varepsilon} C T_{f}(r+|c|)\\\nonumber&&+\frac{|f(0)|}{2}
\end{eqnarray} holds for $r\not\in E.$\par

Under the condition \eqref{E-8.0'} (it implies $\rho_{2}(f)\leq 1$
according to Remark \ref{R0}(i)), we get
\begin{eqnarray}\label{E-8.3}
\frac{\log T_{f}(r+|c|)(\log (r+|c|))^{\varepsilon}}{r+|c|}=o(1)\
(r\to\infty),
\end{eqnarray}and
\begin{eqnarray}\label{E-8.5}
\frac{\log\log T_{f}(r+|c|)}{\log (r+|c|)}=O(1)\
(r\to\infty).\end{eqnarray}
 \par

Furthermore, by \eqref{E-8.2}, we also have
$$T_{f}\left(r+\frac{r\log r}{\log T_{f}(r)(\log\log T_{f}(r))^{1+\varepsilon}}\right)\leq C T_{f}(r)$$
for all $r\not\in E_c=\{x:\ x-|c|\in E\}$. It follows also from the
condition \eqref{E-8.0'} that for all sufficiently large $r,$
\begin{eqnarray*}
\frac{\log T_{f}(r)(\log
r)^{\varepsilon}}{r}<2^{-1-\varepsilon}|c|^{-1},\,\,\,\,\,\,\,
\frac{\log\log T_{f}(r)}{\log r}\leq 2.\end{eqnarray*} Hence, for
all sufficiently large $r,$
\begin{equation*}\frac{r\log r}{\log T_{f}(r)(\log\log T_{f}(r))^{1+\varepsilon}}
=\frac{r}{\log T_{f}(r)(\log r)^{\varepsilon}}\left(\frac{\log
r}{\log\log T_{f}(r)}\right)^{1+\varepsilon}>|c|.\end{equation*}
Hence, we have
\begin{eqnarray}\label{E-8.4}T_{f}(r+|c|)\leq
T_{f}\left(r+\frac{r\log r}{\log T_{f}(r)(\log\log
T_{f}(r))^{1+\varepsilon}}\right)\leq C T_{f}(r)\end{eqnarray} for
all $r\not\in E_c.$ Therefore, substituting \eqref{E-8.3},
\eqref{E-8.5} and \eqref{E-8.4} into \eqref{E-8.1} we have
\begin{eqnarray*}
m(r, f(x+c)\oslash f(x))
&\leq& 16 |c| o(1)O(1)C^{2} T_{f}(r) +\frac{|f(0)|}{2}\\
&=&o(T_{f}(r))
\end{eqnarray*} holds for all $r\not\in E\cup E_c$ where $\int_{E\cup E_c}\frac{dt}{t\log t}<+\infty.$
\end{proof}

\begin{remark} We do not know how to improve the condition \eqref{E-8.0'} by \eqref{E-8.0} in Theorem \ref{C2} whenever making use of Lemma \ref{L3}. The difficulty we met is how to give well defined functions $p(r)$ and $\varphi(r)$ when applying Lemma \ref{L3}.
\end{remark}

\section{Second main theorem with tropical hypersurfaces and defect relation}

In this section we mainly obtain a second main theorem with tropical hypersurfaces and give a tropical version of the Shiffman's conjecture on defect relation. From which, the tropical version of Cartan-Nocha's second main due to Korhonen and Tohge \cite{korhonen-tohge-2016} is improved and extended. Recall that in 2016, Korhonen and Tohge obtained a  second main theorem for tropical holomorphic curves with essentially tropical hyperplanes, which is an analogue of the Cartan second main theorem in Classical Nevanlinna theory \cite{cartan}. The tropical holomorphic curve is assumed with growth condition of hyperorder strictly less than one.\par

\begin{theorem}\cite{korhonen-tohge-2016}\label{T0} Let $q$ and $n$ be positive integers with $q>n,$ and let $\epsilon>0.$ Given $n+1$ tropical entire functions $g_{0}, \ldots, g_{n}$ without common roots, and linearly independent in Gondran-Minoux sense, let $q+1$ tropical linear combinations $f_{0}, \ldots, f_{q}$ of the $g_{j}$ over the semiring $\mathbb{R}_{\max}$ be defined by
$$f_{k}(x)=a_{0k}\odot g_{0}(x)\oplus a_{1k}\odot g_{1}(x)\oplus \cdots\oplus a_{nk}\odot g_{n}(x), 0\leq k\leq q.$$
Let $\lambda=ddg (\{f_{n+1}, \ldots, f_{q}\})$ and $$L=\frac{f_{0}\odot f_{1}\odot\cdots\odot f_{n}\odot f_{n+1}\odot\cdots\odot f_{q}}{C_{o}(f_{0}, f_{1}, \ldots, f_{n})}\oslash.$$
 If the tropical holomorphic curve $g$ of $\mathbb{R}$ into $\mathbb{TP}^{n}$ with reduced representation $g=(g_{0}, \ldots, g_{n})$ is of hyper-order $\rho_{2}(g)=\rho_{2}<1,$ then $$(q-n-\lambda)T_{g}(r)\leq N(r, 1_{o}\oslash L)-N(r, L)+o\left(\frac{T_{g}(r)}{r^{1-\rho_{2}-\epsilon}}\right),$$
where $r$ approaches infinity outside an exceptional set of finite logarithmic measure.
\end{theorem}

Consider a homogeneous tropical polynomial with degree $d$ in $n$ dimensional tropical projective space $\mathbb{TP}^{n}$ of the form
\begin{eqnarray*} P(x)&=&\bigoplus_{I_{i}\in\mathcal{J}_{d}}c_{I{i}}\odot x^{ I_{i}}\\&=&\bigoplus_{i_{0}+i_{1}+\ldots+i_{n}=d} c_{i_{0}, i_{1}, \ldots, i_{n}}\odot x_{0}^{\odot i_0}\odot x_{1}^{\odot i_1}\cdots \odot x_{n}^{\odot i_n},
\end{eqnarray*}where $\mathcal{J}_{d}$ is the set of all $I_{i}=(i_{0}, i_{1}, \ldots, i_{n})\in\mathbb{N}_{0}^{n+1}$ with $\# I_{i}=i_{0}+i_1+\ldots+i_{n}=d.$
The (homogeneous) \textbf{tropical hypersurface} $V_{P}$ in $\mathbb{TP}^{n}$ is the set of zeros (roots) $x=(x_{0}, x_{1}, \ldots, x_{n})$ of $P(x),$  that is, the graph of $P$ is nonlinear at these points (corner locus). In particular, $V_{P}$ is called a tropical hyperplane whenever $d=1.$ For more general definition of tropical hypersurfaces associated to a tropical Laurent polynomial, please see \cite[Definition 3.6]{mikhalkin}. It is shown that $V_{P}$ is the set of points where more than one monomial of $P$ reaches its  maximal value \cite[Proposition 3.3]{mikhalkin}.\par

Set $M:=(_d^{n+d})-1.$  For any $I_{i}=(i_{0}, \ldots, i_{n})\in\mathcal{J}_d,$ $i\in\{0,1,\ldots, M\},$ denote $f^{I_{i}}:=f_0^{\odot i_{0}}\odot\cdots \odot f_n^{\odot i_{n}}.$ Then one can see that the composition function \begin{eqnarray*}P(f)&:=&P\circ f=\bigoplus_{i_{0}+i_{1}+\ldots+i_{n}=d} c_{i_{0}, i_{1}, \ldots, i_{n}}\odot f_{0}^{\odot i_0}\odot f_{1}^{\odot i_1}\cdots f_{n}^{\odot i_n}\\
&=&\bigoplus_{i=0}^{M} c_{I_{i}}\odot f^{I_{i}}\end{eqnarray*} for a tropical holomorphic curve $f:=[f_{0}, \ldots, f_{n}]: \mathbb{R}\rightarrow \mathbb{TP}^{n}$ and tropical hypersurface $V_{P}$ is a tropical algebraical combination of $f_{0}, \ldots, f_{n}$ in the Gondran-Minoux sense. From which, we may also regard $P\circ f$ as a tropical linear combination of $f^{I_{0}}, f^{I_{1}}, \ldots, f^{I_{M}}$  in the Gondran-Minoux sense. From this viewpoint, we introduce some definitions similarly as in Section 2.\par

\begin{definition}\label{D1} Tropical meromorphic functions $f_{0}, \ldots, f_{n}$ are algebraically dependent (respectively independent) in the Gondran-Minoux sense if and only if  $f^{I_{0}}, \ldots, f^{I_{M}}$ are linearly dependent (respectively independent) in the Gondran-Minoux sense. \end{definition}

\begin{definition} Let $G=\{f_{0}, \ldots, f_{n}\}(\neq \{0_{o}\})$ be a set of tropical entire functions, algebraically independent in the Gondran-Minoux sense, and denote \begin{equation*}\mathcal{\hat{L}}_{G}=span<f^{I_{0}}, \ldots, f^{I_{M}}>=\left\{\bigoplus_{k=0}^{M}a_{k}\odot f^{I_{k}}: (a_{0}, \ldots, a_{M})\in\mathbb{R}^{M+1}_{\max}\right\}
\end{equation*} to  be their algebraic span. The collection $G$ is called the algebraic spanning basis of $\mathcal{\hat{L}}_{G}.$ The dimension of $\mathcal{\hat{L}}_{G}$ is defined by
$$\dim(\mathcal{\hat{L}}_{G})=\max\{\hat{\ell}(F): F\in \mathcal{\hat{L}}_{G}\setminus\{0_{o}\}\},$$ where $\hat{\ell}(F)$ is the shortest length of the representation of $F\in \mathcal{\hat{L}}_{G}\setminus\{0_{o}\}$ defined by $$\hat{\ell}(F)=\min\{j\in\{1,\ldots, M+1\}: F=\bigoplus_{i=1}^{j}a_{k_i}\odot f^{I_{k_i}}\}$$ where $a_{k_i}\in\mathbb{R}$ with integers $0\leq k_1<k_2<\cdots <k_j\leq M.$\end{definition}

Note that usually the dimension of the tropical algebraic span space of $G$ may not be M+1, which is different from the classical linear algebraic. If $\hat{\ell}(F)=M+1$ for a tropical algebraic combination $F$ of $f_{0}, \ldots, f_{n},$ then $F$ is said to be complete, that is, the coefficients $a_{k}$ in any expression of $F$ of the form $F=\bigoplus_{k=0}^{M}a_{k}\odot f_{k}$ must satisfy $a_{k}\in \mathbb{R}$ for all $k\in\{0, \ldots, M\}$ such that $\mathcal{\hat{L}}_{G}=M+1.$\par

Furthermore, the tropical Casorati determinant $\tilde{C}(f)=C(f^{I_{0}}, \ldots, f^{I_{M}})$ is given as
\begin{eqnarray*} \hat{C}_{o}(f)=C_{o}(f^{I_{0}}, \ldots, f^{I_{M}})=\bigoplus \overline{f^{I_{0}}}^{[\pi(0)]}\odot \overline{f^{I_{1}}}^{[\pi(1)]}\odot\cdots\odot\overline{f^{I_{M}}}^{[\pi(M)]},
\end{eqnarray*} where the sum is taken over all permutations $\{\pi(0), \ldots, \pi(M)\}$ of $\{0, 1, \ldots, M\}.$ Clearly, when $d=1,$ we have $\hat{C}_{o}(f)=C_{o}(f).$ \par

Recall that in the classical algebraical geometry, if the image of a holomorphic curve $f=[f_{0}: f_{1}: \ldots: f_{n}]: \mathbb{C}\rightarrow\mathbb{P}^{n}(\mathbb{C})$ cannot be contained in any hypersurface (respectively hyperplane) in $\mathbb{P}^{n}(\mathbb{C}),$ then say that $f$ is algebraically (respectively linearly) nondegenerated that means that the entire functions $f_{0}, f_{1}, \ldots, f_{n}$ are algebraically (respectively linearly) independently. We give a similar definition for tropical holomorphic curves.\par

\begin{definition}
Let $f=[f_{0}: f_{1}: \ldots: f_{n}]:\mathbb{R}\rightarrow\mathbb{TP}^{n}$ be a tropical holomorphic curve. If for any tropical hypersurface (respectively hyperplane) $V_{P}$ in $\mathbb{TP}^{n}$ defined by a homogeneous tropical polynomial $P$ in $\mathbb{R}^{n+1},$ $f(\mathbb{R})$ is not a subset of $V_{P},$ then we say that $f$ is tropical algebraically (respectively linearly) nondegenerated.
\end{definition}

Then we give a relationship between tropical algebraically nondegenerated and algebraically independently in Grondran-Minoux sense.\par

\begin{proposition}\label{P1} A tropical holomorphic curve $f: \mathbb{R}\rightarrow\mathbb{TP}^{n}$ with reduced representation $f=(f_{0}, \ldots, f_{n})$ is tropical algebraically (respectively linearly) nondegenerated if and only if $f_{0}, \ldots, f_{n}$ are algebraically (respectively linearly) independently in the Gondran-Minoux sense.
\end{proposition}

\begin{proof} Assume that $f=(f_{0}, \ldots, f_{n})$ is tropical algebraically (respectively linearly) nondegenerated, this means that for any hypersurface (respectively hyperplane) $V_{P}$ in $\mathbb{TP}^{n}$ defined by a homogeneous tropical polynomial $P$ in $\mathbb{R}^{n+1},$ $f(\mathbb{R})\not\subset V_{P}.$ Now if $f_{0}, \ldots, f_{n}$ are algebraically (respectively linearly) dependently in the Gondran-Minoux sense, then there exist two nonempty disjoint subsets $I$ and $J$ of $K:=\{0, \ldots, M\}$ such that $I\cup J=K$ and
\begin{eqnarray*} \bigoplus_{i\in I}a_{i}\odot f_{0}^{\odot i_{0}}\odot\cdots\odot f_{n}^{\odot i_{n}}=\bigoplus_{j\in J}a_{j}\odot f_{0}^{\odot j_{0}}\odot\cdots\odot f_{n}^{\odot j_{n}}
\end{eqnarray*} where  $M=(_d^{n+d})-1,$  all $a_{i}, a_{j}\in\mathbb{R}_{\max}$ and $i_{0}+\ldots+i_{n}=j_{0}+\ldots+j_{n}\in \mathcal{J}_{d}.$
Hence, it gives a homogeneous tropical polynomial $\tilde{P}(x)=\bigoplus_{k=0}^{M} a_{k}\odot x_{0}^{\odot k_{0}}\odot\cdots\odot x_{n}^{\odot k_{n}}$ with degree $d=k_{0}+\ldots+k_{n}\in \mathcal{J}_{d}$ such that
\begin{eqnarray*} \tilde{P}(f)&=&\bigoplus_{k=0}^{M} a_{k}\odot f_{0}^{\odot k_{0}}\odot\cdots\odot f_{n}^{\odot k_{n}}\\
&=&\bigoplus_{i\in I}a_{i}\odot f_{0}^{\odot i_{0}}\odot\cdots\odot f_{n}^{\odot i_{n}}\\&=&\bigoplus_{j\in J}a_{j}\odot f_{0}^{\odot j_{0}}\odot\cdots\odot f_{n}^{\odot j_{n}}.
\end{eqnarray*} This implies that
$(f_{0}(x), \ldots, f_{n}(x))$ are points of $V_{\tilde{P}}$ for all $x\in \mathbb{R}.$ We obtain a contradiction. Hence $f_{0}, \ldots, f_{n}$ must be algebraically (respectively linearly) dependently in the Gondran-Minoux sense.\par

Now assume that $f_{0}, \ldots, f_{n}$ are algebraically (respectively linearly) independently in the Gondran-Minoux sense. If $f$ is tropical algebraically (respectively linearly) degenerated, then there exists one hypersurface (respectively hyperplane) $V_{P}$ in $\mathbb{TP}^{n}$ defined by a homogeneous tropical polynomial $P(x)=\bigoplus_{k=0}^{M} a_{k}\odot x_{0}^{\odot k_{0}}\odot\cdots\odot x_{n}^{\odot k_{n}}$ with degree $d=k_{0}+\ldots+k_{n}\in \mathcal{J}_{d}$ in $\mathbb{R}^{n+1},$ such that $f(\mathbb{R})\subset V_{P}.$ This means that for all $x\in \mathbb{R},$ $(f_{0}(x), \ldots, f_{n}(x))$ are roots of $P(x).$  Since $V_{P}$ consists of some lines (including segment lines and half lines) or located points at which the maximum is attained by two or more of the tropical monomials in $P.$  According to the continuity of the tropical curve $f,$  all image points $(f_{0}(x), \ldots, f_{n}(x))$ are located in continuous piecewise lines or  located points in $V_{P}$ on which the maximum can be taken at least twice in $$P\circ f=\bigoplus_{k=0}^{M} a_{k}\odot f_{0}^{\odot k_{0}}\odot\cdots\odot f_{n}^{\odot k_{n}}.$$ Denote by $\{L_{j}\}_{j=1}^{m}$ the set of all the lines or located points of $f(\mathbb{R})\subset V_{P}.$ Then we know that for each $L_{j}$ $(j\in\{1, \ldots, m\}),$ the maximum can be attained by two tropical monomials of $P\circ f,$ say that, given by
\begin{eqnarray*} a_{\hat{j}^1{}}\odot f_{0}^{\odot \hat{j}_{0}^{1}}\odot\cdots\odot f_{n}^{\odot \hat{j}_{n}^{1}}=a_{\hat{j}^{2}}\odot f_{0}^{\odot \hat{j}_{0}^{2}}\odot\cdots\odot f_{n}^{\odot \hat{j}_{n}^{2}}
\end{eqnarray*} where $\{\hat{j}^{1}, \hat{j}^{2}\}\subset\{0, 1, \ldots, M\},$  $\hat{j}_{0}^{1}+\ldots+\hat{j}_{n}^{1}=\hat{j}_{0}^{2}+\ldots+\hat{j}_{n}^{2}\in\mathcal{J}_{d},$ both $a_{\hat{j}^{1}}$ and $a_{\hat{j}^{2}}$ are not equal to $0_{o}.$ Hence We get that  for all $x\in \mathbb{R},$
\begin{eqnarray*}\bigoplus_{i\in I} b_{i}\odot f_{0}^{\odot i_{0}}\odot\cdots\odot f_{n}^{\odot i_{n}}
=\bigoplus_{j\in J} b_{j}\odot f_{0}^{\odot j_{0}}\odot\cdots\odot f_{n}^{\odot j_{n}}
\end{eqnarray*} where $I\cup J=\{0, 1,\ldots, M\},$ $I\cap J=\emptyset$ and \begin{eqnarray*}b_{k}=\left\{
                                               \begin{array}{lll}
                                                  a_{\hat{j}^{1}}, & \hbox{$k=\hat{j}^{1}\in I;$} \\
                                                  a_{\hat{j}^{2}}, & \hbox{$k=\hat{j}^{2}\in J;$}\\
                                                 0_{o}, & \hbox{otheres.}
                                               \end{array}
                                             \right.
\end{eqnarray*} This contradicts the assumption that $f_{0}, \ldots, f_{n}$ are algebraically (respectively linearly) independently in the Gondran-Minoux sense. Hence $f$ is tropical algebraically (respectively linearly) nondegenerated.
\end{proof}

For giving a new tropical first main theorem, we introduce the tropical Weil function, tropical proximity function for a tropical holomorphic map $f$ with respect to the tropical hypersurface $V_{P}$ with degree $d,$ similarly as in Classical Nevanlinna theory.\par

\begin{definition}  Let $f:\mathbb{R}\rightarrow\mathbb{TP}^{n}$ be a tropical holomorphic map,  let $V_{P}$ be a tropical hypersurface with degree $d$ defined by a homogeneous polynomial $P$ of degree $d$ and let $a$ be the vector defined by the polynomial $P.$ The tropical proximity function $m_{f}(r, V_{P})$ of $f$ with respect to $V_{P}$ is defined as
\begin{eqnarray*}m_{f}(r, V_{P}):=\frac{1}{2}\sum_{\sigma=\pm 1}\lambda_{V_{P}}(f(\sigma r)),\end{eqnarray*}
where $\lambda_{V_{P}}(f(\sigma r))$ means the tropical Weil function defined by $$\lambda_{V_{P}}(f(x)):=\frac{\|f(x)\|^{\odot d}\odot \|a\|^{\odot d}}{P(f)(x)}\oslash.$$\end{definition}

Note that $P(f)$ is a tropical holomorphic function on $\mathbb{R}$ which thus doesn't have any pole. Hence by the tropical Jensen formula, we have
\begin{eqnarray*} N(r, 1_{o}\oslash P(f))=\frac{1}{2}\sum_{\sigma=\pm 1}P(f)(\sigma r)-P(f)(0).
\end{eqnarray*}  Now one can easily deduce the following first main theorem for tropical hypersurfaces by the definitions of tropical characteristic function, counting function and approximation function. \par

\begin{theorem}({\bf First Main Theorem})\label{T4} If $f(\mathbb{R})\not\subset V_{P},$ then
\begin{eqnarray*}m_{f}(r, V_{P})+N(r, 1_{o}\oslash P(f))=d T_{f}(r)+O(1).\end{eqnarray*}
\end{theorem}

In Section 6 we will discuss the difference of the counting function $N(r, 1_{o}\oslash P(f))$ from the counting function $N(r, \frac{1_{o}}{f\oplus a})$ for a tropical meromorphic function intersecting a value $a$ in $\mathbb{TP}^{1}.$ Hence the new tropical first main theorem (Theorem \ref{T4}) we obtained is different from the tropical first main theorem  \eqref{E1.1} due to Laine and Tohge \cite{laine-tohge} (see also \cite{halburd-southall}, \cite[Theorem 3.5]{korhonen-laine-tohge}).\par

Now we give our main result on the tropical second main theorem with tropical hypersurfaces as follows, in which the growth of tropical holomorphic curve $f$ is assumed as $\limsup_{r\rightarrow\infty}\frac{\log T_{f}(r) }{r}=0.$ It is surprised that whenever $\lambda=0,$ all inequalities become equalities. This is very different from the classical Nevanlinna theory.\par

\begin{theorem}({\bf Second Main theorem with tropical hypersurfaces})\label{T1} Let $q$ and $n$ be positive integers with $q\geq n.$ Let the tropical holomorphic curve $f: \mathbb{R}\rightarrow\mathbb{TP}^{n}$ be tropical algebraically nondegerated. Assume that tropical hypersurfaces $V_{P_{j}}$ are defined by homogeneous tropical polynomials $P_{j}$ $(j=1, \ldots, q)$ with degree $d_{j},$ respectively, and $d$ are the least common number of $d_{1}, \ldots, d_{q}.$ Let $M=(_d^{n+d})-1.$
If $\lambda=ddg (\{P_{M+2}\circ f, \ldots, P_{q}\circ f\})$ and $$\limsup_{r\rightarrow\infty}\frac{\log T_{f}(r)}{r}=0,$$  then \begin{eqnarray*}&&
(q-M-1-\lambda)T_{f}(r)\\&\leq&\sum_{j=1}^{q}\frac{1}{d_{j}}N(r, \frac{1_{o}}{P_{j}\circ f}\oslash)-\frac{1}{d}N(r, \frac{1_{o}}{C_{o}(P_{1}^{\odot \frac{d}{d_{1}}}\circ f, \ldots, P_{M+1}^{\odot \frac{d}{d_{M+1}}}\circ f)}\oslash)+o(T_{f}(r))\\
&=&\sum_{j=M+2}^{q}\frac{1}{d_{j}}N(r, \frac{1_{o}}{P_{j}\circ f}\oslash)+o(T_{f}(r))\\
&\leq&(q-M-1)T_{f}(r)+o(T_{f}(r))
\end{eqnarray*}
where $r$ approaches infinity outside an exceptional set of zero upper density measure. In the special case whenever $\lambda=0,$
\begin{eqnarray*}&&
(q-M-1)T_{f}(r)\\&=&\sum_{j=1}^{q}\frac{1}{d_{j}}N(r, \frac{1_{o}}{P_{j}\circ f}\oslash)-\frac{1}{d}N(r, \frac{1_{o}}{C_{o}(P_{1}^{\odot \frac{d}{d_{1}}}\circ f, \ldots, P_{M+1}^{\odot \frac{d}{d_{M+1}}}\circ f)}\oslash)+o(T_{f}(r))\\
&=&\sum_{j=M+2}^{q}\frac{1}{d_{j}}N(r, \frac{1_{o}}{P_{j}\circ f}\oslash)+o(T_{f}(r))
\end{eqnarray*}where $r$ approaches infinity outside an exceptional set of zero upper density measure.
\end{theorem}

\begin{definition} The defect of a tropical holomorphic curve $f: \mathbb{R}\rightarrow\mathbb{TP}^{n}$ intersecting a tropical hypersurface $V_{P}$ given by a tropical polynomial $P$ with degree $d$ on $\mathbb{R}^{n+1}$ is defined by
\begin{eqnarray*}
\delta_{f}(V_{P}):=\liminf_{r\rightarrow\infty} \frac{m_{f}(r, V_{P})}{d T_{f}(r)}=1-\limsup_{r\rightarrow\infty} \frac{N(r, \frac{1_{o}}{P\circ f}\oslash)}{d T_{f}(r)}.
\end{eqnarray*}
\end{definition} Then by Theorem \ref{T1}, we obtain immediately the following defect relation.\par

\begin{theorem}({\bf Defect relation}) \label{T5}Let $q$ and $n$ be positive integers with $q\geq n.$ Let the tropical holomorphic curve $f: \mathbb{R}\rightarrow\mathbb{TP}^{n}$ be tropical algebraically  nondegenerated. Assume that $P_{j}$ $(j=1, \ldots, q)$ are homogeneous tropical polynomials with degree $d_{j},$ and $d$ are the least common number of $d_{1}, \ldots, d_{q}.$ Let $M=(_d^{n+d})-1.$
If $\lambda=ddg (\{P_{M+2}\circ f, \ldots, P_{q}\circ f\})$ and $$\limsup_{r\rightarrow\infty}\frac{\log T_{f}(r)}{r}=0,$$  then
\begin{equation*}\sum_{j=1}^{q}\delta_{f}(V_{P_{j}})\leq M+1+\lambda,\,\, \mbox{and}\,\, \sum_{j=M+2}^{q}\delta_{f}(V_{P_{j}})\leq \lambda.\end{equation*} In special case whenever $\lambda=0,$ we get that $\delta_{f}(V_{P_{j}})=0$ for each $j\in\{M+2, \ldots,  q\}.$
\end{theorem}

This theorem can be regarded as a tropical version of the Shiffman's conjecture \cite{shiffman} on defect relation in the classical Nevanlinna theory which says that $\sum_{j=1}^{q}\delta_{f}(Q_{j})\leq n+1$ holds for a algebraically nondegenerated holomorphic curve $f:\mathbb{C}\rightarrow\mathbb{P}^{n}(\mathbb{C})$ and hypersurfaces $\{Q_{j}\}_{j=1}^{q}$ with the same degree $d$ located in general position in $\mathbb{P}^{n}(\mathbb{C}).$ The Shiffman's conjecture was completely proved by Ru \cite{ru2004}. For further study, we propose also the tropical version of the Griffith's conjecture \cite{griffiths1972} (see also \cite{shiffman, ru2004, siu, hu-yang, biancofiore}) (that is, $\sum_{j=1}^{q}\delta_{f}(Q_{j})\leq \frac{n+1}{d}$) as follows, which is partially proved by Biancofiore \cite{biancofiore} for a class of holomorphic curves, by Siu \cite{siu} with $n+1=3$ using meromorphic connections, and by Hu-Yang \cite{hu-yang} solving a weaker form for a special class of holomorphic curves, respectively.\par

\begin{conjecture} Let $q$ and $n$ be positive integers with $q\geq n.$ Let the tropical holomorphic curve $f: \mathbb{R}\rightarrow\mathbb{TP}^{n}$ be tropical algebraically nondegenerated. Assume that $P_{j}$ $(j=1, \ldots, q)$ are homogeneous tropical polynomials with degree $d_{j},$ and $d$ are the least common number of $d_{1}, \ldots, d_{q}.$ Let $M=(_d^{n+d})-1.$
If $\lambda=ddg (\{P_{M+2}\circ f, \ldots, P_{q}\circ f\})$ and $\limsup_{r\rightarrow\infty}\frac{\log T_{f}(r) }{r}=0,$  then
\begin{equation*}\sum_{j=1}^{q}\delta_{f}(V_{P_{j}})\leq \frac{n+1+\lambda}{d}.\end{equation*}
\end{conjecture}

By Theorem \ref{T1} we get a new tropical version of Cartan-Nochka's second main theorem as follows. Whenever $d=d_{j}=1,$ we have $M=(_d^{n+d})-1=n;$  and from the case (i) in the proof of Theorem \ref{T1} in the next section we have $$N(r, 1_{o}\oslash L)-N(r, L)=\sum_{j=1}^{q}N(r, 1_{o}\oslash P_{j}\circ f)-N(r, 1_{o}\oslash C_{o}(P_{1}\circ f, \ldots, P_{n+1}\circ f)),$$ where $$L:=\frac{(P_{1}\circ f)\odot (P_{2}\circ f)\odot\cdots\odot (P_{n+1}\circ f)\odot (P_{n+2}\circ f)\odot\cdots\odot (P_{q}\circ f)}{C_{o}(P_{1}\circ f, P_{2}\circ f, \ldots, P_{n+1}\circ f)}\oslash.$$ Hence Theorem \ref{T0} is just a special case for tropical hyperplanes, that is, $d=d_{j}=1$ for all $j=1, 2, \ldots, q.$ Furthermore, we adopt the idea from the proof of \cite[Corollary 7.2]{korhonen-tohge-2016} and can deal with the ramification term $N(r, \frac{1_{o}}{C_{o}(P_{1}\circ f, \ldots, P_{n+1}\circ f)}\oslash).$ From this, it might be possible to consider the truncated form of tropical Cartan second main theorem in future. \par

\begin{corollary}({\bf Second main theorem with tropical hyperplanes}) \label{C0} Let $q$ and $n$ be positive integers with $q\geq n.$ Let the tropical holomorphic curve $f: \mathbb{R}\rightarrow\mathbb{TP}^{n}$ be tropical linearly nondegerated. Assume that $V_{P_{j}}$ are tropical hyperplanes in $\mathbb{TP}^{n}$ defined by homogeneous tropical polynomials $P_{j}$ $(j=1, \ldots, q)$ with degree $1,$ respectively.
If $\lambda=ddg (\{P_{n+2}\circ f, \ldots, P_{q}\circ f\})$ and $$\limsup_{r\rightarrow\infty}\frac{\log T_{f}(r) }{r}=0,$$  then
\begin{eqnarray*}&&
(q-n-1-\lambda)T_{f}(r)\\&\leq&N(r, 1_{o}\oslash L)-N(r, L)+o(T_{f}(r))\\&=&\sum_{j=1}^{q}N(r, \frac{1_{o}}{P_{j}\circ f}\oslash)-N(r, \frac{1_{o}}{C_{o}(P_{1}\circ f, \ldots, P_{n+1}\circ f)}\oslash)+o(T_{f}(r))\\
&\leq&\sum_{j=n+2}^{q}N(r, \frac{1_{o}}{P_{j}\circ f}\oslash)+o(T_{f}(r))\\
&\leq&(q-n-1)T_{f}(r)+o(T_{f}(r))
\end{eqnarray*}
where $r$ approaches infinity outside an exceptional set of zero upper density measure. In special case whenever $\lambda=0,$ we have
\begin{eqnarray*}&&
(q-n-1)T_{f}(r)\\&=&N(r, 1_{o}\oslash L)-N(r, L)+o(T_{f}(r))\\&=&\sum_{j=1}^{q}N(r, \frac{1_{o}}{P_{j}\circ f}\oslash)-N(r, \frac{1_{o}}{C_{o}(P_{1}\circ f, \ldots, P_{n+1}\circ f)}\oslash)+o(T_{f}(r))\\
&=&\sum_{j=n+2}^{q}N(r, \frac{1_{o}}{P_{j}\circ f}\oslash)+o(T_{f}(r))
\end{eqnarray*}
where $r$ approaches infinity outside an exceptional set of zero upper density measure.
\end{corollary}

\section{Proof of Theorem \ref{T1}}

(i). We first assume that $d_{j}=d$ holds for all $j=1, \ldots, q$ and \begin{eqnarray*} P_{j}(x)&=&\bigoplus_{i_{j0}+\ldots+i_{jn}=d} a_{i_{j0}, \ldots, i_{jn}}\odot x_{1}^{\odot i_{j0}}\odot\cdots\odot x_{n+1}^{\odot i_{jn}}\\
&=&\bigoplus_{k=0}^{M} a_{I_{jk}}x^{I_{jk}}.\end{eqnarray*} Take $f^{I_i}=f_{0}^{\odot i_{0}}\odot\cdots\odot f_{n}^{\odot i_{n}}$ $(i=0, \ldots, M)$ which are still tropical entire functions on $\mathbb{R}.$ Since $f=(f_{0}, f_{1}, \ldots, f_{n})$ is tropical algebraically nondegerated, it follows from Proposition \ref{P1} and Definition \ref{D1} that  $f=(f_{0}, \ldots, f_{n})$ is algebraically independent in the Gondran-Minoux sense, and thus $\tilde{F}=(f^{I_0}, f^{I_1}, \ldots, f^{I_M})$ is linearly independent in the Gondran-Minoux sense. Denote $g_{j-1}:=P_{j}\circ f$ for all $j=1,2, \ldots, q.$ By the properties of tropical Casorati determinant,
\begin{eqnarray*}
C_{o}(g_{0}, \ldots, g_{M})=g_{0}\odot \overline{g_{0}}\odot\cdots\odot\overline{g_{0}}^{[M]}\odot C_{o}(1_{o}, g_{1}\oslash g_{0}, \ldots, g_{M}\oslash g_{0}).
\end{eqnarray*}Take
\begin{equation*}
 \tilde{L}:=\frac{g_{0}\odot \overline{g_{1}}\odot\cdots\odot\overline{g_{M}}^{[M]}\odot g_{M+1}\odot\cdots\odot g_{q-1}}{C_{o}(g_{0}, g_{1}, \ldots, g_{M})}\oslash
\end{equation*}
and
\begin{eqnarray*}
\psi:=g_{M+1}\odot\cdots\odot g_{q-1},
\end{eqnarray*} then
\begin{eqnarray*}
\psi=\tilde{L}\odot K,
\end{eqnarray*} where\begin{eqnarray*}
K=C_{o}(1_{o}, g_{1}\oslash g_{0}, \ldots, g_{M}\oslash g_{0})\odot(\overline{g_{0}}\oslash\overline{g_{1}})\odot\cdots\odot(\overline{g_{0}}^{[M]}\oslash\overline{g_{M}}^{[M]}).
\end{eqnarray*}

We denote the $g_{\nu},$ $(M+1\leq\nu\leq q-1),$ to be
\begin{equation*}
g_{\nu}:=\bigoplus_{j\in S_{\nu}} t_{j\nu}\odot f^{I_{j\nu}}(x)=\max_{j\in S_{\nu}}\{t_{j\nu}+f^{I_{j\nu}}(x)\}, t_{j\nu}\in \mathbb{R},
\end{equation*} for  index sets $S_{\nu}\subset\{0, 1, \ldots, M\}$ with cardinality $\#S_{\nu}(\leq M+1).$ Then we have \begin{eqnarray}\label{E-a}
&&\sum_{\nu=M+1, \# S_{\nu}=M+1}^{q-1}\left(\frac{1}{2}\sum_{\sigma=\pm 1}g_{\nu}(\sigma r)\right)\\\nonumber
&=&\sum_{\nu=M+1, \# S_{\nu}=M+1}^{q-1}\frac{1}{2}\left(\max_{j\in S_{\nu}}\{t_{j\nu}+f^{I_{\j\nu}}(r)\}+\max_{j\in S_{\nu}}\{t_{j\nu}+f^{I_{\j\nu}}(-r)\}\right)\\\nonumber
&\geq& \sum_{\nu=M+1, \# S_{\nu}=M+1}^{q-1}\frac{1}{2}\left(\max_{j\in S_{\nu}}\{f^{I_{\j\nu}}(r)\}+\max_{j\in S_{\nu}}\{f^{I_{\j\nu}}(-r)\}\right)\\\nonumber&&+\sum_{\nu=M+1, \# S_{\nu}=M+1}^{q-1}\min_{j\in S_{\nu}}\{t_{j\nu}\}\\\nonumber
&=& \sum_{\nu=M+1, \# S_{\nu}=M+1}^{q-1}\frac{1}{2}\left(\max_{j_{\nu 0}+\ldots+j_{\nu n}=d}\{j_{\nu 0}f_{0}(r)+\ldots+j_{\nu n}f_{n}(r)\}\right.\\\nonumber&&
+\left.\max_{j_{\nu 0}+\ldots+j_{\nu n}=d}\{j_{\nu 0}f_{0}(-r)+\ldots+j_{\nu n}f_{n}(-r)\}\right)+\sum_{\nu=M+1, \# S_{\nu}=M+1}^{q-1}\min_{j\in S_{\nu}}\{t_{j\nu}\}\\\nonumber
&\geq& \sum_{\nu=M+1, \# S_{\nu}=M+1}^{q-1}\frac{1}{2}\left(d\max_{j=0}^{n}\{f_{j}(r)\}
+d\max_{j=0}^{n}\{f_{j}(-r)\}\right)\\\nonumber&&+\sum_{\nu=M+1, \# S_{\nu}=M+1}^{q-1}\min_{j\in S_{\nu}}\{t_{j\nu}\}.\end{eqnarray}

The condition $\lambda=ddg (\{P_{M+2}\circ f, \ldots, P_{q}\circ f\})$ which is the number of its non-complete elements means that there exist $q-M-1-\lambda$ complete elements in the set $\{P_{M+2}\circ f, \ldots, P_{q}\circ f\}.$ Since $g_{\nu}$ $(M+1\leq\nu\leq q-1)$ are piecewise linear entire functions on $\mathbb{R},$ there exist $\alpha_{\nu}, \beta_{\nu}\in\mathbb{R}$ and an interval $[r_{1}, r_{2}]\subset\mathbb{R}$ containing the origin such that $r_{1}<r_{2}$ and \begin{equation*}
g_{\nu}(x)=\alpha_{\nu}x+\beta_{\nu}
\end{equation*} for all $x\in[r_{1}, r_{2}].$ Then we get that
\begin{eqnarray*}
g_{\nu}(0)=\beta_{\nu}=\max_{j\in S_{\nu}}\{t_{j\nu}+f^{I_{j\nu}}(0)\}.
\end{eqnarray*} If define \begin{equation*}
h_{\nu}(x):=\alpha_{\nu}(x)+\beta_{\nu}
\end{equation*} for all $x\in \mathbb{R},$ then by the convexity of the graph of $g_{\nu}$ we get that
\begin{equation*}
g_{\nu}(x)\geq h_{\nu}(x)
\end{equation*}for all $x\in\mathbb{R}.$ Hence
\begin{equation*}
\frac{1}{2}\sum_{\sigma=\pm 1}g_{\nu}(\sigma x)\geq  \frac{1}{2}\sum_{\sigma=\pm 1}h_{\nu}(\sigma x)=\beta_{\nu}.
\end{equation*} This gives that \begin{eqnarray}\label{E-b}
&&\frac{1}{2}\sum_{\sigma=\pm 1}\psi(\sigma r)\\\nonumber&=&
\sum_{\nu=M+1}^{q-1}\left(\frac{1}{2}\sum_{\sigma=\pm 1}g_{\nu}(\sigma r)\right)\\\nonumber
&=&  \sum_{\nu=M+1,\# S_{\nu}=M+1}^{q-1}\left(\frac{1}{2}\sum_{\sigma=\pm 1}g_{\nu}(\sigma r)\right)+\sum_{\nu=M+1, \# S_{\nu}<M+1}^{q-1}\left(\frac{1}{2}\sum_{\sigma=\pm 1}g_{\nu}(\sigma r)\right)\\\nonumber
&\geq&\sum_{\nu=M+1, \# S_{\nu}=M+1}^{q-1}\left(\frac{1}{2}\sum_{\sigma=\pm 1}g_{\nu}(\sigma x)\right)+\sum_{\nu=M+1, \# S_{\nu}<M+1}^{q-1}\beta_{\nu}.\end{eqnarray}

According to the definition of tropical Cartan characteristic function,
\begin{eqnarray*}
T_{f}(r)+\max_{j=0, 1, \ldots, n}\{f_{j}(0)\}=\frac{1}{2}\max\{f_{0}(r), \ldots, f_{n}(r)\}+\frac{1}{2}\max\{f_{0}(-r), \ldots, f_{n}(-r)\}.
\end{eqnarray*}
Then we get from \eqref{E-a} that
\begin{eqnarray}\label{E-c}
&&\sum_{\nu=M+1, \# S_{\nu}=M+1}^{q-1}\left(\frac{1}{2}\sum_{\sigma=\pm 1}g_{\nu}(\sigma r)\right)\\\nonumber
&\geq&\sum_{\nu=M+1, \# S_{\nu}=M+1}^{q-1}d\left(T_{f}(r)+\max_{j=0}^{n}\{f_{j}(0)\}\right)+\sum_{\nu=M+1, \# S_{\nu}=M+1}^{q-1}\min_{j\in S_{\nu}}\{t_{j\nu}\}\\\nonumber
&=&(q-M-1-\lambda)d\left(T_{f}(r)+\max_{j=0}^{n}\{f_{j}(0)\}\right)+\sum_{\nu=M+1, \# S_{\nu}=M+1}^{q-1}\min_{j\in S_{\nu}}\{t_{j\nu}\}.
\end{eqnarray}
Therefore, it follows from \eqref{E-b} and \eqref{E-c} that
\begin{eqnarray*}
\frac{1}{2}\sum_{\sigma=\pm 1}\psi_{\nu}(\sigma r)
&\geq&(q-M-1-\lambda)d\left(T_{f}(r)+\max_{j=0}^{n}\{f_{j}(0)\}\right)\\&&+\sum_{\nu=M+1, \# S_{\nu}=M+1}^{q-1}\min_{j\in S_{\nu}}\{t_{j\nu}\}+\sum_{\nu=M+1, \# S_{\nu}<M+1}^{q-1}\beta_{\nu}.
\end{eqnarray*}
This gives an inequality of characteristic function $T_{f}(r)$ as follows
\begin{eqnarray}\label{E-d}&&
(q-M-1-\lambda)T_{f}(r)\\\nonumber &\leq& \frac{1}{d}\left(\frac{1}{2}\sum_{\sigma=\pm 1}\psi(\sigma r)\right)+\frac{1}{d}\sum_{\nu=M+1, \# S_{\nu}=M+1}^{q-1}\min_{j\in S_{\nu}}\{t_{j\nu}\}\\\nonumber&&+\frac{1}{d}\sum_{\nu=M+1, \# S_{\nu}<M+1}^{q-1}\beta_{\nu}+ (q-M-1-\lambda)\max_{j=0}^{n}\{f_{j}(0)\}.
\end{eqnarray}

Next we need obtain an estimation on the first term of the right side of the above inequality.  By the tropical Jensen's theorem and the definition of $\psi,$ we deduce
\begin{eqnarray*}&&
\frac{1}{2}\sum_{\sigma=\pm 1}\psi(\sigma r)\\&=&\frac{1}{2}\sum_{\sigma=\pm 1}\tilde{L}(\sigma r)+\frac{1}{2}\sum_{\sigma=\pm 1}K(\sigma r)\\
&=&N(r, 1_{o}\oslash \tilde{L})-N(r, \tilde{L})+\tilde{L}(0)+\frac{1}{2}\sum_{\sigma=\pm 1}K^{+}(\sigma r)-\frac{1}{2}\sum_{\sigma=\pm 1}(-K)^{+}(\sigma r)\\
&=&N(r, 1_{o}\oslash \tilde{L})-N(r, \tilde{L})+\tilde{L}(0)+m(r, K)-m(r, 1_{o}\oslash K)\\
&\leq&N(r, 1_{o}\oslash \tilde{L})-N(r, \tilde{L})+\tilde{L}(0)+m(r, K).
\end{eqnarray*}
Denote \begin{equation*}
L=\frac{g_{0}\odot g_{1}\odot\cdots\odot g_{M}\odot\cdots\odot g_{q-1}}{C_{o}(g_{0}, g_{1}, \ldots, g_{M})}\oslash,
 \end{equation*} which gives
 \begin{equation*}
\tilde{L}=L\odot \frac{\overline{g_{1}}\odot \overline{g_{2}}^{[2]}\odot\cdots\odot \overline{g_{M}}^{[M]}}{g_{1}\odot g_{2}\odot\cdots\odot g_{M}}\oslash.
 \end{equation*}
Then  the above inequalities give
\begin{eqnarray}\label{E-e}&&
\frac{1}{2}\sum_{\sigma=\pm 1}\psi(\sigma r)\\\nonumber&\leq&N(r, 1_{o}\oslash \tilde{L})-N(r, \tilde{L})+L(0)+\sum_{j=0}^{M}g_{j}(j)-\sum_{j=0}^{M}g_{j}(0)+m(r, K).
\end{eqnarray}
We will estimate $m(r, K)$ below. Since $g_{j}\oslash g_{0}$ are tropical meromorphic functions all $j\in\{1, 2, \ldots, M\},$ we have \begin{eqnarray*}
&&T_{g_{j}\oslash g_{0}}(r)\\&=&\frac{1}{2}\sum_{\sigma=\pm 1}\max\{g_{j}(\sigma r), g_{0}(\sigma r)\}-\max\{g_{j}(0), g_{0}(0)\}\\
&=& \frac{1}{2}\sum_{\sigma=\pm 1}\max_{j_{0}+\ldots+j_{n}=d}\{c_{j_0, \ldots, j_{n}}+j_{0}f_{0}(\sigma r)+\ldots+j_{n}f_{n}(\sigma r)\}-\max\{g_{j}(0), g_{0}(0)\}\\
&=& \frac{1}{2}\sum_{\sigma=\pm 1}\max_{j_{0}+\ldots+j_{n}=d}\{j_{0}f_{0}(\sigma r)+\ldots+j_{n}f_{n}(\sigma r)\}\\&&-\max\{g_{j}(0), g_{0}(0)\}+\max_{j_{0}+\ldots+j_{n}=d}\{c_{j_0, \ldots, j_{n}}\}\\
&\leq&\frac{1}{2}\sum_{\sigma=\pm 1}d\max\{f_{0}(\sigma r), f_{1}(\sigma r),\ldots, f_{n}(\sigma r)\}\\&&-\max\{g_{j}(0), g_{0}(0)\}+\max_{j_{0}+\ldots+j_{n}=d}\{c_{j_0, \ldots, j_{n}}\}\\
&\leq&dT_{f}(r)+d\max\{f_{0}(0), f_{1}(0), \ldots, f_{n}(0)\}-\max\{g_{j}(0), g_{0}(0)\}\\&&+\max_{j_{0}+\ldots+j_{n}=d}\{c_{j_0, \ldots, j_{n}}\}.
\end{eqnarray*}
This implies that \begin{eqnarray*} \limsup_{r\rightarrow \infty}\frac{\log T_{g_{j}\oslash g_{0}}(r)}{r}\leq\limsup_{r\rightarrow\infty}\frac{\log T_{f}(r) }{r}=0,\end{eqnarray*}
and then by Lemma \ref{L2} we get that  for any $k\in \mathbb{N},$
\begin{eqnarray*}
T_{\overline{g_{j}\oslash g_{0}}^{[k]}}(r)=(1+\varepsilon(r))
T_{g_{j}\oslash g_{0}}(r)=T_{f}(r)+o(T_{f}(r))
\end{eqnarray*} holds for all $r\not\in E$ with $\overline{dens} E=0$ (Throughout this proof, $E$ always means having the property $\overline{dens} E=0$). Therefore, for any $k\in \mathbb{N}$ \begin{eqnarray*}\limsup_{r\rightarrow \infty}\frac{\log T_{\overline{g_{j}\oslash g_{0}}^{[k]}}(r)}{r}\leq\limsup_{r\rightarrow\infty}\frac{\log T_{f}(r) }{r}=0.\end{eqnarray*}
Note that
\begin{eqnarray*}
K&=&C_{o}(1_{o}, g_{1}\oslash g_{0}, \ldots, g_{M}\oslash g_{0})\odot(\overline{g_{0}}\oslash\overline{g_{1}})\odot\cdots\odot(\overline{g_{0}}^{[M]}\oslash\overline{g_{M}}^{[M]})\\
&=&\frac{\bigoplus(\overline{g_{1}}^{[\pi(0)]}\oslash \overline{g_{0}}^{[\pi(0)]})\odot \ldots\odot (\overline{g_{M}}^{[\pi(M)]}\oslash \overline{g_{0}}^{[\pi(M)]})}{(\overline{g_{1}}\oslash\overline{g_{0}})\odot\cdots\odot(\overline{g_{M}}^{[M]}\oslash\overline{g_{0}}^{[M]})}\oslash\\
&=&\bigoplus\left(\frac{\left(\overline{g_{1}\oslash g_{0}}\right)^{[\pi(0)]}}{\overline{g_{1}\oslash g_{0}}}\oslash\right) \odot \ldots\odot \left(\frac{\left(\overline{g_{M}\oslash g_{0}}\right)^{[\pi(M)]}}{\overline{g_{M}\oslash g_{0}}^{[M]}}\oslash\right)
\end{eqnarray*}
where the tropical sum is taken over all permutations $\{\pi(0), \ldots, \pi(M)\}$ of the set $\{0, 1, \ldots, M\}.$ Now by Theorem \ref{L1}, we deduce that \begin{eqnarray}\label{E-f}
m(r, K)=o(T_{f}(r))
\end{eqnarray}holds for all $r\not\in E$ with $\overline{dens}E=0.$\par

Therefore, it follows from \eqref{E-d}, \eqref{E-e} and \eqref{E-f} that
\begin{eqnarray}\label{E-g}
(q-M-1-\lambda)T_{f}(r)&\leq&\frac{1}{d}N(r, 1_{o}\oslash \tilde{L})-\frac{1}{d}N(r, \tilde{L})+o(T_{f}(r))
\end{eqnarray}for all $r\not\in E$ with $\overline{dens}E=0.$\par

The next step is to estimate $N(r, 1_{o}\oslash \tilde{L})$ and $N(r, \tilde{L}).$ Note that
\begin{equation*}
\tilde{L}=L\odot \frac{\overline{g_{1}}\odot \overline{g_{2}}^{[2]}\odot\cdots\odot \overline{g_{M}}^{[M]}}{g_{1}\odot g_{2}\odot\cdots\odot g_{M}}\oslash
 \end{equation*} and that $g_{1}, \ldots, g_{M}$ are tropical entire functions. Then by the tropical Jensen formula,
 \begin{eqnarray}\label{E-h}
&&N(r, 1_{o}\oslash\tilde{L})-N(r, \tilde{L})\\\nonumber
&=&\frac{1}{2}\sum_{\sigma=\pm 1} \tilde{L}(\sigma r) -\tilde{L}(0)\\\nonumber
&=&\frac{1}{2}\sum_{\sigma=\pm 1} L(\sigma r)-L(0)+\sum_{j=1}^{M}\frac{1}{2}\left(\sum_{\sigma=\pm 1}\overline{g_{j}}^{[j]}(\sigma r)+\overline{g_{j}}^{[j]}(0)\right)\\\nonumber&&-\sum_{j=1}^{M}\frac{1}{2}\left(\sum_{\sigma=\pm 1}g_{j}(\sigma r)+g_{j}(0)\right)\\\nonumber
&=&N(r, 1_{o}\oslash L)-N(r, L)+\sum_{j=1}^{M}\left(N(r, 1_{o}\oslash\overline{g_{j}}^{[j]})-N(r, 1_{o}\oslash g_{j})\right)\\\nonumber
&\leq&N(r, 1_{o}\oslash L)-N(r, L)+\sum_{j=1}^{M}\left(N(r+jc, 1_{o}\oslash g_{j})-N(r, 1_{o}\oslash g_{j})\right),
 \end{eqnarray}where the last inequality is deduced from the geometric meaning between $N(r, 1_{o}\oslash \overline{g}_{j}^{[j]})$ and $N(r+jc, 1_{o}\oslash g_{j}).$ Using the tropical Jensen formula again, we deduce that
\begin{eqnarray*}
&&N(r, 1_{o}\oslash g_{j})\\&=& \frac{1}{2}\sum_{\sigma=\pm 1}g_{j}(\sigma r)-g_{j}(0)\\
&=& \frac{1}{2}\sum_{\sigma=\pm 1}\max_{j_{0}+\ldots+j_{n}=d}\{c_{j_0, \ldots, j_{n}}+j_{0}f_{0}(\sigma r)+\ldots+j_{n}f_{n}(\sigma r)\}-g_{j}(0)\\
&=& \frac{1}{2}\sum_{\sigma=\pm 1}\max_{j_{0}+\ldots+j_{n}=d}\{j_{0}f_{0}(\sigma r)+\ldots+j_{n}f_{n}(\sigma r)\}\\&&-g_{j}(0)+\max_{j_{0}+\ldots+j_{n}=d}\{c_{j_0, \ldots, j_{n}}\}\\
&\leq&\frac{1}{2}\sum_{\sigma=\pm 1}d\max\{f_{0}(\sigma r), f_{1}(\sigma r),\ldots, f_{n}(\sigma r)\}\\&&-g_{j}(0)+\max_{j_{0}+\ldots+j_{n}=d}\{c_{j_0, \ldots, j_{n}}\}\\
&\leq&dT_{f}(r)+d\max\{f_{0}(0), f_{1}(0), \ldots, f_{n}(0)\}-g_{j}(0)\\&&+\max_{j_{0}+\ldots+j_{n}=d}\{c_{j_0, \ldots, j_{n}}\},
\end{eqnarray*}
This implies \begin{eqnarray*}
\limsup_{r\rightarrow\infty} \frac{\log N(r, 1_{o}\oslash g_{j})}{r}\leq \limsup_{r\rightarrow\infty} \frac{\log T_{f}(r) }{r}=0.\end{eqnarray*}
Hence by Lemma \ref{L2},  \begin{eqnarray}\label{E1}
N(r+jc, 1_{o}\oslash g_{j})-N(r, 1_{o}\oslash g_{j})=\varepsilon(r)N(r, 1_{o}\oslash g_{j})= o(T_{f}(r))
\end{eqnarray}holds for $r\not\in E$ with $\overline{dens}E=0.$ Therefore we get from \eqref{E-h} and \eqref{E1} that  \begin{eqnarray*}
N(r, 1_{o}\oslash\tilde{L})-N(r, \tilde{L})
\leq N(r, 1_{o}\oslash L)-N(r, L)+o(T_{f}(r))
 \end{eqnarray*}holds for $r\not\in E$ with $\overline{dens}E=0.$
Combining this with \eqref{E-g} gives
\begin{eqnarray}\label{E-i}
(q-M-1-\lambda)T_{f}(r)&\leq&\frac{1}{d}N(r, 1_{o}\oslash L)-\frac{1}{d}N(r, L)+o(T_{f}(r)).
\end{eqnarray}for all $r\not\in E$ with $\overline{dens}E=0.$\par

Note that $g_{j}$ $(j=0, \ldots, q-1)$ and $C_{o}(g_{0}, \ldots, g_{M})$ are all tropical entire functions. Then according to the definition of $L,$ we can get from the tropical Jensen formula that
\begin{eqnarray}\label{E-j}&&
N(r, 1_{o}\oslash L)-N(r, L)\\\nonumber&=&\frac{1}{2}\sum_{\sigma=\pm 1}L(\sigma r)-L(0)\\\nonumber
&=&\sum_{j=0}^{q-1}\left(\frac{1}{2}\sum_{\sigma=\pm 1}g_{j}(\sigma r)-g_{j}(0)\right)\\\nonumber
&&-\left(\frac{1}{2}\sum_{\sigma=\pm 1}C_{o}(g_{0}, \ldots, g_{M})(\sigma r)-C_{o}(g_{0}, \ldots, g_{M})(0)\right)\\\nonumber
&=&\sum_{j=0}^{q-1}N(r, 1_{o}\oslash g_{j})-N(r, 1_{o}\oslash C_{o}(g_{0}, \ldots, g_{M})).
\end{eqnarray}
Now combining \eqref{E-i} and \eqref{E-j}, we get the form of the second main theorem that  \begin{eqnarray}\label{E6.4}
&&(q-M-1-\lambda)T_{f}(r)\\\nonumber
&\leq&\frac{1}{d}\sum_{j=0}^{q-1}N(r, 1_{o}\oslash g_{j})-\frac{1}{d}N(r, 1_{o}\oslash C_{o}(g_{0}, \ldots, g_{M}))+o(T_{f}(r)).
\end{eqnarray}for all $r\not\in E$ with $\overline{dens}E=0.$\par

Now we estimate $N(r, 1_{o}\oslash C_{o}(g_{0}, \ldots, g_{M})).$ According to the definition of tropical Casorati determinant, we have
\begin{eqnarray}\label{E6.1}
&&C_{o}(g_{0}, \ldots, g_{M})\\\nonumber
&=&\bigoplus(\overline{g_{0}}^{[\pi(0)]}\odot \cdots \odot \overline{g_{M}}^{[\pi(M)]})\\\nonumber
&=&\left\{\bigoplus\left[\left(\overline{g_{0}}^{[\pi(0)]}\odot \cdots \odot \overline{g_{M}}^{[\pi(M)]}\right)\oslash\left(g_{0} \odot \cdots \odot g_{M}\right)\right]\right\}+g_{0}\odot \cdots \odot g_{M}\\\nonumber
&=&\bigoplus\left[\left(\overline{g_{0}}^{[\pi(0)]}\oslash g_{0}\right) \odot \cdots \odot \left(\overline{g_{M}}^{[\pi(M)]}\oslash g_{M}\right)\right]+ g_{0}\odot \cdots \odot g_{M},\end{eqnarray} where the sum is taken over all permutations $\{\pi(0), \ldots, \pi(M)\}$ of $\{0, 1, \ldots, M\}.$ If denote $$D=\bigoplus\left[\left(\overline{g_{0}}^{[\pi(0)]}\oslash g_{0}\right) \odot \cdots \odot \left(\overline{g_{M}}^{[\pi(M)]}\oslash g_{M}\right)\right],$$ then \begin{eqnarray*} D\geq \left(\overline{g_{0}}^{[\pi(0)]}\oslash g_{0}\right) \odot \cdots \odot \left(\overline{g_{M}}^{[\pi(M)]}\oslash g_{M}\right)\end{eqnarray*} for any  permutation $\{\pi(0), \ldots, \pi(M)\}$ of $\{0, 1, \ldots, M\}.$ By the tropical Jensen formula and \eqref{E1},
\begin{eqnarray}\label{E6.2}&&\frac{1}{2}\sum_{\sigma=\pm 1} D(\sigma r)-D(0)\\\nonumber
&\geq& \sum_{j=0}^{M}\left(\frac{1}{2}\sum_{\sigma=\pm 1} (\overline{g_{j}}^{[\pi(j)]}\oslash g_{j})(\sigma r)\right)-D(0)\\\nonumber
&=&\sum_{j=0}^{M}\left(\frac{1}{2}\sum_{\sigma=\pm 1} \overline{g_{j}}^{[\pi(j)]}(\sigma r)- \frac{1}{2}\sum_{\sigma=\pm 1}g_{j}(\sigma r)\right)-D(0)\\\nonumber
&=& \sum_{j=0}^{M}\left(N(r, 1_{o}\oslash(\overline{g_{j}}^{[\pi(j)]}))-N(r, 1_{o}\oslash g_{j})\right)-D(0)+\sum_{j=0}^{M}\left(\overline{g_{j}}^{[\pi(j)]}(0)-g_{j}(0)\right)\\\nonumber
&=& \sum_{j=0}^{M}\left(N(r+jc, 1_{o}\oslash g_{j})-N(r, 1_{o}\oslash g_{j})\right)-D(0)+\sum_{j=0}^{M}\left(\overline{g_{j}}^{[\pi(j)]}(0)-g_{j}(0)\right)\\\nonumber
&=&o(T_{f}(r)),
\end{eqnarray} holds for $r\not\in E$ with $\overline{dens}E=0.$
Hence using the tropical Jensen formula again, it gives by \eqref{E6.1} and \eqref{E6.2} that
\begin{eqnarray}\label{E6.3}&&
N(r, 1_{o}\oslash C_{o}(g_{0},\ldots, g_{M}))\\\nonumber
&=&\frac{1}{2}\sum_{\sigma =\pm 1}C_{o}(g_{0}, \ldots, g_{M})(\sigma r)-C_{o}(g_{0},\ldots, g_{M})(0)\\\nonumber
&=&\frac{1}{2}\sum_{\sigma=\pm 1} D(\sigma r)-D(0)+\sum_{j=0}^{M}\left(\frac{1}{2}\sum_{\sigma=\pm 1}g_{j}(\sigma r)-g_{j}(0)\right)\\\nonumber
&=&\frac{1}{2}\sum_{\sigma=\pm 1} D(\sigma r)-D(0)+\sum_{j=0}^{M}N(r, 1_{o}\oslash g_{j})\\\nonumber
&\geq& o(T_{f}(r))+\sum_{j=1}^{M+1}N(r, \frac{1_{o}}{P_{j}\circ f}\oslash)
\end{eqnarray}holds for $r\not\in E$ with $\overline{dens}E=0.$ Submitting \eqref{E6.3} into \eqref{E6.4} gives that
\begin{eqnarray*}&&(q-M-1-\lambda)T_{f}(r)\\
&\leq&\frac{1}{d}\sum_{j=0}^{q-1}N(r, 1_{o}\oslash g_{j})-\frac{1}{d}N(r, 1_{o}\oslash C_{o}(g_{0}, \ldots, g_{M}))+o(T_{f}(r))\\
&\leq&\frac{1}{d}\sum_{j=M+2}^{q}N(r, \frac{1_{o}}{P_{j}\circ f}\oslash)+o(T_{f}(r))\\
\end{eqnarray*}where $r$ approaches infinity outside an exceptional set of zero upper density measure.\par

(ii). We now consider general case whenever the degree of homogeneous polynomials $P_{j}$ $(j=1, \ldots, q)$ are $d_{j}$ respectively. Assume that
\begin{eqnarray*} P_{j}(x)&=&\bigoplus_{i_{j0}+\ldots+i_{jn}=d_{j}} a_{i_{j0}, \ldots, i_{jn}}\odot x_{1}^{\odot i_{j0}}\odot\cdots\odot x_{n+1}^{\odot i_{jn}}\\
&=&\max_{i_{j0}+\ldots+i_{jn}=d_{j}}\{a_{i_{j0}, \ldots, i_{jn}}+ i_{j0}x_{1}+\cdots+ i_{jn}x_{n+1}\}.\end{eqnarray*}
Then \begin{eqnarray*} P_{j}^{\odot \frac{d}{d_{j}}}(x)&=&\frac{d}{d_{j}}P_{j}(x)\\
&=&\frac{d}{d_{j}}\bigoplus_{i_{j0}+\ldots+i_{jn}=d_{j}} a_{i_{j0}, \ldots, i_{jn}}\odot x_{1}^{\odot i_{j0}}\odot\cdots\odot x_{n+1}^{\odot i_{jn}}\\
&=&\max_{i_{j0}+\ldots+i_{jn}=d_{j}}\{\frac{d}{d_{j}}a_{i_{j0}, \ldots, i_{jn}}+ \frac{d}{d_{j}}i_{j0}x_{1}+\cdots+ \frac{d}{d_{j}}i_{jn}x_{n+1}\}\\
&=&\bigoplus_{\frac{d}{d_{j}}i_{j0}+\ldots+\frac{d}{d_{j}}i_{jn}=d
}(\frac{d}{d_{j}}a_{i_{j0}, \ldots, i_{jn}})\odot x_{1}^{\odot \frac{d}{d_{j}}i_{j0}}\odot\cdots\odot x_{n+1}^{\odot \frac{d}{d_{j}}i_{jn}}.
\end{eqnarray*} Thus all $P_{j}^{\odot \frac{d}{d_{j}}}(x)$ are of degree $d.$  Furthermore, we can see that if $x_{0}$ is a root of $P_{j}\circ f$ with multiplicity $\omega_{P_{j}\circ f}(x_{0})>0,$  then $x_{0}$ should be also a root of $P_{j}^{\odot \frac{d}{d_{j}}}\circ f$ $(=\frac{d}{d_{j}}P_{j}\circ f)$ with multiplicity $\omega_{P_{j}^{\odot \frac{d}{d_{j}}}\circ f}(x_{0})=\frac{d}{d_{j}}\omega_{P_{j}\circ f}(x_{0})>0.$ The inverse is also true. This implies that \begin{eqnarray*}N(r, \frac{1_{o}}{P_{j}^{\odot \frac{d}{d_{j}}}\circ f}\oslash)=\frac{d}{d_{j}}N(r, \frac{1_{o}}{P_{j}\circ f}\oslash).\end{eqnarray*} Hence by the conclusion (i), we have
\begin{eqnarray*}&&
(q-M-1-\lambda)T_{f}(r)\\&\leq&\frac{1}{d}\sum_{j=1}^{q}N(r, \frac{1_{o}}{P_{j}^{\odot \frac{d}{d_{j}}}\circ f}\oslash)-\frac{1}{d}N(r, \frac{1_{o}}{C_{o}(P_{1}^{\odot \frac{d}{d_{1}}}\circ f, \ldots, P_{M+1}^{\odot \frac{d}{d_{M+1}}}\circ f)}\oslash)+o(T_{f}(r))\\
&\leq&\frac{1}{d}\sum_{j=M+2}^{q}N(r, \frac{1_{o}}{P_{j}^{\odot \frac{d}{d_{j}}}\circ f}\oslash)+o(T_{f}(r))\\
&=&\sum_{j=M+2}^{q}\frac{1}{d_{j}}N(r, \frac{1_{o}}{P_{j}\circ f}\oslash)+o(T_{f}(r))
\end{eqnarray*}where $r$ approaches infinity outside an exceptional set of finite upper density measure. By the first main theorem (Theorem \ref{T4}) we have
$$N(r, \frac{1_{o}}{P_{j}\circ f}\oslash)\leq d_{j} T_{f}(r)$$ for all $j=1, \ldots, q.$ Therefore, the theorem is proved immediately.\par

\section{Tropical Nevanlinna second main theorem and defect relation}
In this section we mainly give new versions of tropical Nevanlinna second main theorem which are very different from before and then obtain the defect relation. Recall that the first version of second main theorem for tropical meromorphic functions was obtained by Laine and Tohge \cite{laine-tohge}, which is an analogue of the Nevanlinna second main theorem for meromorphic functions on the complex plane in the Classical Nevanlinna theory \cite{nevanlinna}. Note that the equality \eqref{E4.7} below was written as an inequality form in the original statement \cite{laine-tohge},  and holds from the fact that $N(r, 1_{o}\oslash(f\oplus a_{j}))\leq T(r, f)+O(1)$ for all $j\in\{1, \ldots, q\}$ according to the first main theorem \eqref{E1.1}. \par

\begin{theorem}\cite{laine-tohge}\label{T2} If $f$ is a nonconstant tropical meromorphic function of hyperorder $\rho_{2}(f)<1,$ if $\varepsilon>0,$ and $q(\geq 1)$ distinct values $a_{1}, \ldots, a_{q}\in\mathbb{R}$ satisfying \begin{eqnarray}
\label{E4.1} \max\{a_{1}, \ldots, a_{q}\}<\inf\{f(\alpha): \omega_{f}(\alpha)<0\},
\end{eqnarray} and \begin{eqnarray}\label{E4.6}\inf\{f(\beta): \omega_{f}(\beta)>0\}>-\infty.\end{eqnarray} Then \begin{eqnarray}\label{E4.7}
qT(r,f)=\sum_{j=1}^{q}N(r, 1_{o}\oslash(f\oplus a_{j}))+o(\frac{T(r, f)}{r^{1-\rho_{2}(f)-\varepsilon}})
\end{eqnarray} for all $r$ outside of an exceptional set of finite logarithmic measure.\end{theorem}

Recently, Korhonen and Tohge \cite[Corollary 7.2]{korhonen-tohge-2016} said that from Theorem \ref{T0} they improved Theorem \ref{T2}  by dropping the assumption \eqref{E4.6}. However, we  find that there exists one gap. In \cite[Page 722]{korhonen-tohge-2016}, it was said that from the assumption \eqref{E4.1} (that is, (7.4) in \cite{korhonen-tohge-2016}) it follows that "the roots of $g_{0}$ are exactly the poles of $a_{k}\oplus f,$ counting multiplicity, for all $k=1,2\ldots, q.$" This is very true! But they continued to assert that by $f_{k}=(a_{k-1}\odot g_{0})\oplus (1_{o}\odot g_{1}),$
\begin{eqnarray*}
N(r, 1_{o}\oslash f_{k})=N(r, 1_{o}\oslash(f\oplus a_{k-1}))\,\,\,\,\,    \mbox{(see the (7.12) in \cite{korhonen-tohge-2016}})
\end{eqnarray*}for all $k=2,\ldots, q+1.$ In fact, this is not true! Firstly, $N(r, 1_{o}\oslash(f\oplus a_{k-1}))$ means the roots of $f\oplus a_{k-1}$ but not poles of $f\oplus a_{k-1};$ secondly, the roots of $g_{0}$ (which is the poles of $f=g_{1}\oslash g_{0}$ as in the statement in \cite{korhonen-tohge-2016}) is not equal to the roots of $f_{k}=(a_{k-1}\odot g_{0})\oplus (1_{o}\odot g_{1}),$ and hence it is not true of $N(r, 1_{o}\oslash g_{0})=N(r, 1_{o}\oslash f_{k}).$ Therefore, it is not reasonable enough to assert that $N(r, 1_{o}\oslash f_{k})=N(r, 1_{o}\oslash(f\oplus a_{k-1})).$ Below we will consider the relationship of the two counting functions to explain this.\par

Let $f=[f_{0}: f_{1}]:\mathbb{R}\rightarrow\mathbb{TP}^{1}$ be a tropical nonconstant meromorphic function, and let $a=[a_{1}: a_{0}]$ be a value of $\mathbb{TP}^{1}$ which defining a tropical polynomial $P(x)=(a_{0}\odot x_{0})\oplus (a_{1}\odot x_{1})$ on $\mathbb{R}^{2}.$ Then the polynomial $P(x)$  gives a tropical hyperplane $V_{P}=\{a\}.$ We can see that
\begin{eqnarray*}
P\circ f(x)&=&(a_{0}\odot f_{0}(x))\oplus (a_{1}\odot f_{1}(x))\\&=&((a_{0}-a_{1})\oplus (f_{1}(x)-f_{0}(x)))+(a_{1}+f_{0}(x))\\
&=&(a\oplus f(x))\odot(a_{1}+f_{0}(x)).
\end{eqnarray*} From this equalities, we only know \begin{eqnarray}\label{E90}N(r, 1_{o}\oslash (P\circ f))&\leq& N(r, 1_{o}\oslash (f\oplus a))+N(r, 1_{o}\oslash(a_{1}+f_{0}))\\\nonumber
&=&N(r, 1_{o}\oslash (f\oplus a))+N(r, f).\end{eqnarray} If discussing some special cases, then it is easy to get the following proposition.\par

\begin{proposition}\label{P3}
(I). If $f\oplus a\equiv f$ (for example, $a=-\infty=[1_{o}: 0_{o}]$), and thus $P\circ f(x)=f_{1}(x)+a_{1},$ then in this case we have
\begin{eqnarray*} N(r, 1_{o}\oslash (P\circ f))=N(r, 1_{o}\oslash (f_{1}+a_{1}))=N(r, 1_{o}\oslash (f\oplus a)). \end{eqnarray*}

(II). If $f\oplus a\equiv a$ (for example, $a=+\infty=[0_{o}: 1_{o}]$), and thus $P\circ f(x)=a+a_{1}+f_{0}(x),$ then in this case we have
$N(r, 1_{o}\oslash (P\circ f))=N(r, 1_{o}\oslash (a+a_{1}+f_{0}))=N(r, f)$ and $N(r, 1_{o}\oslash (f\oplus a))=0.$
\end{proposition}

By Corollary \ref{C0} (or Theorem \ref{T1}), we get directly a new tropical version of Nevanlinna's second main theorem.\par

\begin{theorem}\label{T3} Assume that $f=[f_{0}: f_{1}]:\mathbb{R}\rightarrow\mathbb{TP}^{1}$ is a nonconstant tropical meromorphic function with $$\limsup_{r\rightarrow\infty}\frac{\log T_{f}(r) }{r}=0,$$ and $a_{j}=[a_{j1}: a_{j0}]$ $(j=1, \ldots, q)$ are distinct values of $\mathbb{TP}^{1}$ which defining tropical polynomials $P_{j}(x)=a_{j0}\odot x_{0}\oplus a_{j1}\odot x_{1}$ on $\mathbb{R}^{2},$ respectively. If $ddg (\{P_{3}\circ f, \ldots, P_{q}\circ f\})=\lambda,$ then
\begin{eqnarray}\label{E7.1}&&(q-2-\lambda)T_{f}(r)\\\nonumber
&\leq&\sum_{j=1}^{q}N\left(r, \frac{1_{o}}{P_{j}\circ f}\oslash\right)-N\left(r, \frac{1_{o}}{C_{o}(P_{1}\circ f, P_{2}\circ f)}\oslash \right)+o\left(T_{f}(r)\right)\\\nonumber
&\leq&\sum_{j=3}^{q}N\left(r, \frac{1_{o}}{P_{j}\circ f}\oslash\right)+o\left(T_{f}(r)\right)\\\nonumber
&\leq&(q-2)T_{f}(r)+o\left(T_{f}(r)\right),\end{eqnarray}
where $r$ approaches infinity outside an exceptional set of zero upper density measure. In special case whenever $\lambda=0,$ \begin{eqnarray}\label{E7.2}&&(q-2)T_{f}(r)\\\nonumber&=&
\sum_{j=1}^{q}N\left(r, \frac{1_{o}}{P_{j}\circ f}\oslash\right)-N\left(r, \frac{1_{o}}{C_{o}(P_{1}\circ f, P_{2}\circ f)}\oslash \right)+o\left(T_{f}(r)\right)\\\nonumber
&=&\sum_{j=3}^{q}N\left(r, \frac{1_{o}}{P_{j}\circ f}\oslash\right)+o\left(T_{f}(r)\right),\end{eqnarray}
where $r$ approaches infinity outside an exceptional set of zero upper density measure.
\end{theorem}

\begin{proof} Due to $f$ tropical meromorphic on $\mathbb{R},$ we may assume $f=f_{1}\oslash f_{0}=[f_{0}: f_{1}]: \mathbb{R}\rightarrow \mathbb{TP}^{1}.$ According to Corollary \ref{C0}, we only need prove that $f$ is linearly independent in Gondran-Minoux sense (or say, tropical linearly nondegenerated).
Otherwise, there exist two nonempty sets $I, J$ with $I\cap J=\emptyset$ and $I\cup J=\{0, 1\}$  such that
\begin{eqnarray*}
\bigoplus_{i\in I} b_{i}\odot f_{i}=\bigoplus_{j\in J} b_{j}\odot f_{j},
\end{eqnarray*} that is, there exist two values $b_{0}, b_{1}\in\mathbb{R}_{\max}$ such that $b_{0}\odot f_{0}=b_{1}\odot f_{1}.$ This means $f=f_{1}\oslash f_{0}=b_{0}\oslash b_{1},$ which contradicts to the assumption that $f$ is nonconstant. \end{proof}

Next we give a sufficient condition about $\lambda=0$ in the above theorem.\par

\begin{proposition}\label{P2} Assume that $f=[f_{0}: f_{1}]:\mathbb{R}\rightarrow\mathbb{TP}^{1}$ is a nonconstant tropical meromorphic function with and $a_{j}=[a_{j1}: a_{j0}]$ $(j=1, \ldots, q)$ are distinct finite values of $\mathbb{R}$ which defining tropical polynomials $P_{j}(x)=a_{j0}\odot x_{0}\oplus a_{j1}\odot x_{1}$ on $\mathbb{R}^{2},$ respectively. If $f\not\equiv(f\oplus a_{j})\not\equiv a_{j}$ for all $j=1,2\ldots, q,$  then $\lambda:=ddg (\{P_{1}\circ f, \ldots, P_{q}\circ f\})=0.$\end{proposition}

\begin{proof} Assume $\lambda>0.$ Then there exists at leat one of $\{P_{1}\circ f, \ldots, P_{q}\circ f\},$ say $P_{k}\circ f,$ satisfying $\ell(P_{k}\circ f)<2$ by the definition of the degree of degeneracy. This implies that either $P_{k}\circ f\equiv a_{k0}\odot f_{0}$ or
$P_{k}\circ f\equiv a_{k1}\odot f_{1}.$ Thus we have either
$(a_{k0}+ f_{0})\oplus (a_{k1}+ f_{1})\equiv a_{k0}+ f_{0}$ or
$(a_{k0}+ f_{0})\oplus (a_{k1}+ f_{1})\equiv a_{k1}+ f_{1},$
which contradict either $f\oplus a_{k}\not\equiv a_{k}$ or $f\oplus a_{k}\not\equiv f$ respectively.
\end{proof}

Then we obtain a tropical Nevanlinna theorem which is very similar to Theorem \ref{T2}. But the counting functions are very different from each other. Furthermore,  here we only need one assumption $f\oplus a_{j}\not\equiv a_{j}$ which means that it demands only at least one point $x_{0}$ such that $f(x_{0})>a_{j}$ for each $a_{j}.$ Thus this assumption is obviously better than the condition \eqref{E4.1} in Theorem \ref{T2} due to Laine and Tohge \cite{laine-tohge} which implies that the value $f(x)$ at each poles should be large strictly than $\max\{a_{1}, \ldots, a_{q}\}$.  \par

\begin{theorem}\label{T6} Assume that $f$ is a nonconstant tropical meromorphic function with $$\limsup_{r\rightarrow\infty}\frac{\log T_{f}(r)}{r}=0,$$ and $a_{j}$ $(j=1, 2, \ldots, q)$ defining  tropical polynomial $P_{j}$ are distinct values of $\mathbb{R}$  such that $f\oplus a_{j}\not\equiv a_{j},$ respectively.  Then
\begin{eqnarray*}qT_{f}(r)=\sum_{j=1}^{q}N\left(r, \frac{1_{o}}{P_{j}\circ f}\oslash\right)+o\left(T_{f}(r)\right)
,\end{eqnarray*}
where $r$ approaches infinity outside an exceptional set of zero upper density measure.
\end{theorem}

\begin{proof} If $f\oplus a_{j}\not\equiv f$ is not satisfied for all $j=1,2, \ldots, q.$ Suppose there exists one nonempty subset $S$ of $\{1, 2, \ldots, q\}$ such that $f\oplus a_{k}\equiv f$ for all $k\in S,$ and $f\oplus a_{i}\not\equiv f$ for all $i\in \{1, 2, \ldots, q\}\setminus S.$ Then for $k\in S,$
we have
\begin{eqnarray*}T_{f}(r)&=&T(r, f)+O(1)\\&=&T(r,1_{o}\oslash f)+O(1)\\&=&T(r,1_{o}\oslash (f\oplus a_{k}))+O(1)\\&=&N(r, 1_{o}\oslash (f\oplus a_{k}))+m(r, -(f\oplus a_{k}))+O(1)\\
&\leq& N(r, 1_{o}\oslash(f\oplus a_{k}))+\max\{-a_{k}, 0\}+O(1)\\
&=& N(r, 1_{o}\oslash(f\oplus a_{k}))+O(1)\end{eqnarray*}
Further, by (II) of Proposition \ref{P3} we have
\begin{eqnarray*}N(r, 1_{o}\oslash(P_{k}\circ f))=N(r, 1_{o}\oslash(f\oplus a_{k})).\end{eqnarray*} Hence
\begin{eqnarray}\label{E7.3}\#ST_{f}(r)&\leq&\sum_{k\in S}N(r, 1_{o}\oslash(P_{k}\circ f))+O(1)\\\nonumber
&\leq&\#ST_{f}(r)+O(1)\end{eqnarray}where the last inequality comes from the tropical first main theorem (that is, Theorem \ref{T4}).\par

On the other hand, by Theorem \ref{T3} and Proposition \ref{P2}, we get that  for $i\in \{1, 2, \ldots, q\}\setminus S,$ \begin{eqnarray}\label{E7.4}(q-\#S)T_{f}(r)&=&\sum_{i\in \{1, 2, \ldots, q\}\setminus S}N\left(r, 1_{o}\oslash(P_{i}\circ f)\right)+o\left(T_{f}(r)\right),\end{eqnarray}
where $r$ approaches infinity outside an exceptional set of zero upper density measure. Therefore the corollary is obtained immediately from \eqref{E7.3} and \eqref{E7.4}.
\end{proof}

By Theorem \ref{T3} and Theorem \ref{T6}, we have defect relation as follows. Define that a value $a$ is called a Picard exceptional value of $f$ if $\delta_{f}(a)=1.$ Hence by (ii) below, we get that all finite real values are the Picard exceptional values of a nonconstant tropical meromorphic function. \par

\begin{theorem}\label{C1} Assume that $f=[f_{0}: f_{1}]:\mathbb{R}\rightarrow\mathbb{TP}^{1}$ is a nonconstant tropical meromorphic function with $\limsup_{r\rightarrow\infty}\frac{\log T_{f}(r)}{r}=0.$ \par

(i) Let $a_{j}=[a_{j1}: a_{j0}]$ $(j=1, \ldots, q)$ be distinct values of $\mathbb{TP}^{1}$ which defining tropical polynomials $P_{j}(x)=a_{j0}\odot x_{0}\oplus a_{j1}\odot x_{1}$ on $\mathbb{R}^{2},$ respectively. If $ddg (\{P_{3}\circ f, \ldots, P_{q}\circ f\})=\lambda,$ then $\sum_{j=3}^{q}\delta_{f}(a_{j})=\lambda$ for all $a_{j}$ $(j=3, 4, \ldots, q).$ In special case whenever $\lambda=0,$ $\delta_{f}(a_{j})=0$ holds for each $a_{j}$ $(j\in\{3,4,\ldots, q\}.$\par

(ii) $$\delta_{f}(a)=0$$ holds for each $a\in \mathbb{R}$ such that $f\oplus a\not\equiv a.$
\end{theorem}

\begin{remark} Let $\alpha$  be a real number with $|\alpha|>1.$ Define a tropical meromorphic  function $e_{\alpha}(x)$ on $\mathbb{R}$ by
$$e_{\alpha}(x):=\alpha^{[x]}(x-[x])+\sum_{j=-\infty}^{[x]-1}=\alpha^{[x]}(x-[x]+\frac{1}{\alpha-1}).$$ Similarly, if $\beta$ is a real number with $|\beta|<1,$ the corresponding tropical function defined as
$$e_{\beta}(x):=\beta^{[x]}(\frac{1}{1-\beta}-x+[x]).$$ By \cite[Proposition 1.22 and Proposition 1.24]{korhonen-laine-tohge}, both $e_{\alpha}(x)$ and $e_{\beta}(x)$ are of hyperorder one and  $\limsup_{r\rightarrow\infty}\frac{\log T_{e_{\alpha}}(r)}{r}=\limsup_{r\rightarrow\infty}\frac{\log T_{e_{\beta}}(r)}{r}=1\neq 0.$ By \cite[Reamrk 1.25]{korhonen-laine-tohge},  we have $$1-\limsup_{r\rightarrow\infty}\frac{N(r, 1_{o}\oslash(e_{\beta(x)}\oplus a))}{T_{e_{\beta}(x)}(r)}\geq \frac{1}{2}>0$$ for all $a<0.$ Then by \eqref{E90} and the tropical first main theorem, we get $\delta_{e_{\beta}}(a)\geq\frac{1}{2}>0$ for all $a<0.$ This means that the assumption $\limsup_{r\rightarrow\infty}\frac{\log T_{f}(r) }{r}=0$ cannot be deleted in (ii) of Theorem \ref{C1}.
\end{remark}

Recall that the Classical Nevanlinna second main theorem for meromorphic functions on the complex plane $\mathbb{C}$ have a truncated form as follows \begin{equation*}
(q-2)T_{f}(r)\leq N^{1)}(r, \frac{1}{f-a_{j}})+o(T_{f}(r))
\end{equation*}holds for all $r$ possibly outside a set with finite linear measure, where $a_{1}, \ldots, a_{q}$ are distinct values in $\mathbb{P}^{1}(\mathbb{C})=\mathbb{C}\cup\{\infty\}$ and $N^{1)}(r,\frac{1}{f-a_{j}})$ is the counting function of zeros of $f-a_{j}$ with multiplicities truncated by one (or say ignoring multiplicities). Now let $f$ be a tropical meromorphic function and let $a$ be a value in $\mathbb{TP}^{1}$ which defining a tropical linear polynomial $P$ in $\mathbb{TP}^{1}.$ Then it might be interesting to consider the truncated counting function $N^{1)}(r, \frac{1_{o}}{P_{j}\circ}\oslash)$ corresponding to $N(r, \frac{1_{o}}{P_{j}\circ f}\oslash)$ ignoring the multiplicities of roots of $P_{j}\circ f$ in our results on the tropical Nevanlinna second main theorem (Theorem \ref{T3} and Theorem \ref{T6}). We give a conjecture as follows. \par

\begin{conjecture} Assume that $f$ is a nonconstant tropical meromorphic function with $\limsup_{r\rightarrow\infty}\frac{\log T_{f}(r)}{r}=0.$\par

(i). Let $a_{j}$ $(j=1, \ldots, q)$ be distinct values of $\mathbb{TP}^{1}$ which defining tropical linear polynomials $P_{j}(x)$ on $\mathbb{R}^{2},$ respectively. If $\lambda=ddg (\{P_{1}\circ f, \ldots, P_{q}\circ f\}),$ then
\begin{eqnarray*}&&(q-2-\lambda)T_{f}(r)
\leq\sum_{j=3}^{q}N^{1)}(r, \frac{1_{o}}{P_{j}\circ f}\oslash)+o\left(T_{f}(r)\right)
,\end{eqnarray*}
where $r$ approaches infinity outside an exceptional set of zero upper density measure.\par

(ii). Let $a_{j}$ $(j=1, \ldots, q)$ be distinct finite values in $\mathbb{R}$ which defining tropical linear polynomials $P_{j}(x)$ on $\mathbb{R}^{2}$ such that $f\oplus a_{j}\not\equiv a_{j},$ respectively. Then
\begin{eqnarray*}&& qT_{f}(r)
=\sum_{j=1}^{q}N^{1)}(r, \frac{1_{o}}{P_{j}\circ f}\oslash)+o\left(T_{f}(r)\right)
,\end{eqnarray*}
where $r$ approaches infinity outside an exceptional set of zero upper density measure.
\end{conjecture}

Moreover, we may give a conjecture on truncated form of tropical Cartan second main theorem corresponding to Corollary \ref{C0} as follows. Here $N^{k)}(r, f)$ ia denoted to be the counting function of poles of $f$ with multiplicities truncated by $k$ (that is, the multiplicity is only counted $k$ whenever its multiplicity is greater than $k$).\par

\begin{conjecture} Let $q$ and $n$ be positive integers with $q\geq n.$ Let the tropical holomorphic curve $f: \mathbb{R}\rightarrow\mathbb{TP}^{n}$ be tropical algebraically nondegerated. Assume that $V_{P_{j}}$ are defined by are homogeneous tropical polynomials $P_{j}$ $(j=1, \ldots, q)$ with degree $1,$ respectively.
If $\lambda=ddg (\{P_{n+2}\circ f, \ldots, P_{q}\circ f\})$ and $\limsup_{r\rightarrow\infty}\frac{\log T_{f}(r) }{r}=0,$  then
\begin{eqnarray*}&&
(q-n-1-\lambda)T_{f}(r)\leq\sum_{j=n+2}^{q}N^{n)}(r, \frac{1_{o}}{P_{j}\circ f}\oslash)+o(T_{f}(r))
\end{eqnarray*}
where $r$ approaches infinity outside an exceptional set of zero upper density measure. In special case whenever $\lambda=0,$ we have
\begin{eqnarray*}&&
(q-n-1)T_{f}(r)=\sum_{j=n+2}^{q}N^{n)}(r, \frac{1_{o}}{P_{j}\circ f}\oslash)+o(T_{f}(r))
\end{eqnarray*}
where $r$ approaches infinity outside an exceptional set of zero upper density measure.
\end{conjecture}

\noindent{\bf Acknowledgement.} The first author would like to thank to Professor Ilpo Laine and Professor Risto Korhonen for introducing the tropical Nevanlinna theory when he had visited in University of Eastern Finland from June 2014 until May 2015.


\begin{thebibliography}{99}

\bibitem{biancofiore} A. Biancofiore, A hypersurface defect relation for a class of meromorphic maps, Trans. Amer. Math. Soc. 270(1982), 47-60.
\bibitem{cartan} H. Cartan, Sur les zeros des combinaisions linearires de $p$ fonctions holomorpes donnees, Mathemctica (Cluj) 7(1933), 80-103.
\bibitem{cherry-ye} W. Cherry and Z. Ye, Nevanlinna's Theory of Value Distribution Second Main Theorem and its error terms, Springer-Verlag, Berlin, 2001.
\bibitem{gondran-minoux-1} M. Gondran and M. Minoux, Graphes et algorithmes, volume 37 of Collection de la Direction
des Etudes et Recherches d\'{e}s Electricit\'{e} de France [Collection of the Department of Studies
and Research of Electricit\'{e} de France]. Paris: \'{E}ditions Eyrolles, 1979.
\bibitem{gondran-minoux-2} M. Gondran, and M. Minoux, Linear algebra in dioids: a survey of recent results, In
Algebraic and combinatorial methods in operations research, volume 95 of North-Holland Math. Stud., pages 147-163. Amsterdam: North-Holland, 1984.

\bibitem{griffiths1972} P. A. Griffiths, Holomorphic mappings: survey of some results and dicussion of open problems, Bull. Amer. Math. Soc. 78(1972), No. 3, 374-382.
\bibitem{halburd-korhonen2007} R. G. Halburd and R. J. Korhonen, Finite-order meromorphic solutions and the discrete Painlev\'{e} equations, Proc. London Math. Soc. 94(2007), No. 2, 443-474.
    \bibitem{halburd-korhonen-tohge2014} R. G. Halburd, R. J. Korhonen and K. Tohge, Holomorphic curves with shift-invariant hyperplane preimages, Trans. Amer. Math. Soc. 366(2014), No. 8, 4267-4298.

\bibitem{halburd-southall} R. G. Halburd and N. J. Southall, Tropical Nevanlinna theory and ultradiscrete equations, Int. Math. Res. Not. IMRN 5(2009), 887-911.
\bibitem{hayman} W. K. Hayman. Meromorphic Functions, Oxford Mathematical Monographs,
Clarendon Press, Oxford, 1964.
\bibitem{hinkkanen}   A. Hinkkanen, A sharp form of Nevanlinna's second fundamental theorem, Invent. Math. 108 (1992), 549-574.

\bibitem{hu-yang} P. C. Hu and C. C. Yang, The second main theorem of holomoprhic curves into projective spaces, Trans. Amer. Math. Soc. 363(2011), No. 12, 6465-6479.

\bibitem{kleene} S. C. Kleene, Representation of events in nerve nets and finite automata,
In Automata studies, Annals of mathematics studies, No. 34, pages 3-41, Princeton University Press, 1956.
\bibitem{korhonen-laine-tohge} R. J. Korhonen, I. Laine, K. Tohge, Tropical Value Distribution Theory and Ultra-Discrete Equations, World Scientific publisheing Co. Pte. Ltd.,  Singapore, 2015.

 \bibitem{korhonen-tohge-2016} R. J. Korhonen and K. Tohge, Second main theorem in the tropical projective space, Advances Math. 298(2016), 693-725.

 \bibitem{laine-liu-tohge} I. Laine, K. Liu and K. Tohge, Tropical variants of some complex analysis results, Ann. Acad. Sci. Fenn. Math. 41(2016),  923-946.

\bibitem{laine-tohge} I. Laine and K. Tohge, Tropical Nevanlinna theory and second main theorem, Proc. London. Math. Soc. 102(2011), No. 5, 883-922.
\bibitem{laine-yang} I. Laine and C. C. Yang, Tropical versions of Clunie and Mohon'ko lemmas, Complex Var. Elliptic Equ. 55(2010), No. 1-3, 237-248.
\bibitem{maclagan-sturmfels} D. Maclagan, B. Sturmfels, Introduction to Tropical Geometry, Graduate Studies inMathematics, Volume 161, American Mathematical Society, Providence, Rhode Island, 2015.
\bibitem{mikhalkin} G. Mikhalkin, Enumerative tropical algebraic geometry in $\mathbb{R}^{2},$ J. Amer. Math. Soc. 18(2005), No. 2, 313-377.
\bibitem{nevanlinna} R. Nevanlinna, Xur Theorie der meromorphen Funktionen, Acta. Math. 46(1925), 1-99.
\bibitem{pin} J. E. Pin, Tropical semirings, Idempotency (Bristol, 1994), 50-69, Publ. Newton Inst., 11,
Cambridge Univ. Press, Cambridge, 1998.

\bibitem{ru2004} M. Ru, A defect relation for holomorphic curves intersecting hypersurfaces, Amer. J. Math. 126(2004), 215-226.
\bibitem{ru} M. Ru, Nevanlinna Theory and its Relation to Diophatine Approximation, Singapore: World Scientific Publishing Co., 2001.
\bibitem{siu} Y. T. Siu, Defect relations for holomorphic maps between spaces of different dimensions, Duke Math. J. 55(1987), No. 1, 213-251.

\bibitem{shiffman} B. Shiffman, Holomorphic curves in algebraic manifolds, Bull. Amer. Math. Soc. 83(1977), 553-568.

\bibitem{ZK}J. Zheng and R. Korhonen, Studies of differences from
the point of view of Nevanlinna theory, Arxiv.org/abs/1806.00212
\end{thebibliography}
\end{document}